\documentclass[10pt,twoside, a4paper, english, reqno]{amsart}
\usepackage[dvips]{epsfig}
\usepackage{amssymb,amsthm,amsmath,latexsym,amscd}
\usepackage{upref,hyperref,color}
\usepackage[usenames,dvipsnames]{xcolor}
\usepackage{ulem,cancel,pgf}
\usepackage{pdfsync}
\theoremstyle{plain}
\newtheorem{thm}{Theorem}[section]
\theoremstyle{plain}
\newtheorem{lem}[thm]{Lemma}
\newtheorem{prop}[thm]{Proposition}

\theoremstyle{definition}
\newtheorem{defi}{Definition}[section]
\newtheorem{rem}[thm]{Remark}

\newenvironment{Assumptions}
{
\setcounter{enumi}{0}

\begin{enumerate}}
{\end{enumerate} }

\newcommand{\eps}{\ensuremath{\varepsilon}}
\newcommand{\R}{\ensuremath{\mathbb{R}}}
 \newcommand{\Grad}{\mathrm{\nabla}}
\newcommand{\con} {\ast}
 \newcommand{\supp}{\ensuremath{\mathrm{supp}\,}}
\newcommand{\goto}{\ensuremath{\rightarrow}}
\newcommand{\grad}{\ensuremath{\nabla}}
\newcommand{\rd}{\ensuremath{\mathbb{R}^d}}
\newcommand{\E}{\ensuremath{\mathbb{E}}}
\def\N{{I\!\!N}}

\def\guy#1{\textcolor{blue!90!black}{#1} }

\numberwithin{equation}{section}
\allowdisplaybreaks

\title[ Degenerate parabolic PDE with L\'{e}vy noise]{On the Cauchy problem of a degenerate parabolic-hyperbolic PDE with L\'{e}vy noise}
\date{}



 \thanks{}
\author[Imran H. Biswas]{Imran H. Biswas}
\address[Imran H. Biswas]{\newline
 Centre for Applicable Mathematics,
 Tata 
 \guy{
 Institute
 }
  of Fundamental Research,
  P.O.\ Box 6503, GKVK Post Office,
  Bangalore 560065, India}
\email[]{imran@math.tifrbng.res.in}

\author[Ananta K. Majee]{Ananta K. Majee}
\address[Ananta K. Majee]{\newline
Mathematisches Institut
Universit\"{a}t T\"{u}bingen
Auf der Morgenstelle 10
D-72076 T\"{u}bingen, Germany }
\email[]{majee@na.uni-tuebingen.de}

\author[Guy Vallet]{Guy Vallet}
\address[Guy Vallet]{\newline 
LMAP UMR- CNRS 5142, IPRA BP 1155, 64013 Pau Cedex, France}
\email[]{guy.vallet@univ-pau.fr}

\thanks{The authors are profoundly thankful for the generous support from IFCAM which allowed them to travel between India $\&$ France and made this collaboration possible. The third author would like to also acknowledge the support of ISIFoR. }

\begin{document}
\begin{abstract} 
In this article we deal with stochastic perturbation of degenerate parabolic partial differential equations (PDEs).  The particular emphasise is on analysing the effect of multiplicative L\'{e}vy noise to such problems and establishing wellposedness by developing a suitable weak entropy solution framework. The proof of existence is based on the vanishing viscosity technique. The uniqueness is settled by interpreting Kruzkov's doubling technique in the presence noise. 

\end{abstract}
\maketitle
\tableofcontents

\section{Introduction}
Let $\big(\Omega, P, \mathcal{F}, \{\mathcal{F}_t\}_{t\ge 0} \big)$ be a filtered probability space satisfying 
 the usual hypothesis 
 \guy{
\textit{ i.e.}
 }
  $\{\mathcal{F}_t\}_{t\ge 0}$ is a right-continuous filtration such that $\mathcal{F}_0$ 
 contains all the $P$-null subsets of $(\Omega, \mathcal{F})$. In addition, let $\big(E, \mathcal{E}, m\big)$ be a $\sigma$-finite measure space and $N(\,dt, \,dz )$ be a Poisson random measure on  $\big(E, \mathcal{E}\big)$ with intensity measure $m(\,dz)$ with respect to the same stochastic basis. The existence and construction of such general notion of Poisson random measure with a given intensity measure are detailed in \cite{peszat}. 
   We are interested in the Cauchy problem for a nonlinear degenerate parabolic stochastic PDE of
 the following type
 \begin{align} du(t,x) -\Delta \phi(u(t,x))\,dt - \mbox{div}_x f(u(t,x)) \,dt=\int_{E}
 \eta(x, u(t,x);z)\tilde{N}(dz,dt),\quad (t,x)\in \Pi_T,\label{eq:levy_stochconservation_laws}
\end{align}
   with the initial condition
   \begin{align}
   \label{initial_cond} u(0,x) = u_0(x), \quad \quad x\in \R^d,
   \end{align} 
   where $\Pi_T= [0,T)\times \R^d$ with $T>0$ fixed, $u(t,x)$ is the unknown random scalar valued function,
   $ F:\R\rightarrow \R^d$ is  given flux function, and $ \tilde{N}(dz,dt)= N(dz,dt)-\, m(dz)\,dt $, the compensated Poisson random measure. Furthermore, $(x, u,z)\mapsto \eta(x, u; z)$ is a real valued  function defined on the domain $\R^d\times \R\times E$ and $\phi:\R \rightarrow \R$ 
is a given non-decreasing Lipschitz continuous function. The stochastic integral in the RHS of \eqref{eq:levy_stochconservation_laws} is defined in the L\'{e}vy-It\^{o} sense. 
\begin{rem}
 Since $\phi$ is a real valued non-decreasing, Lipschitz continuous function, the set \\
 $A=\big\{ r\in \R: \phi^\prime(r)=0 \big\}$ is not empty in general and hence the problem is called  degenerate.
 Even more, $A$ is not negligible either and the problem is strongly degenerate in the sense of \cite{carrillo_1999}.
\end{rem}

\begin{rem}
 The analysis of this paper remains valid if the noise on the RHS of \eqref{eq:levy_stochconservation_laws} is of jump-diffusion type.
 In other words, the same analysis holds if we add a $\sigma(x, u) dW_t$ term in the RHS of  \eqref{eq:levy_stochconservation_laws} where $W_t$ is a
 cylindrical Brownian motion. Moreover, we will carry out our analysis under the structural assumption $E = \mathcal{O}\times \R^*$ where $\mathcal{O}$ is a 
 subset of the Euclidean space. The measure $m$ on $E$ is defined as $\lambda \times \mu$ where $\lambda$ is a Radon measure on $\mathcal{O}$ and $\mu$ is a so 
 called L\'{e}vy measure on $\R^*$. In such a case, the noise of the RHS would be called an impulsive white noise with jump position intensity $\lambda$ and
 jump size intensity $\mu$.  We refer to \cite{peszat} for more on L\'{e}vy sheet and related impulsive white noise. 
 
 \end{rem}
The equation \eqref{eq:levy_stochconservation_laws} becomes a multidimensional deterministic degenerate parabolic-hyperbolic 
  equation if $\eta =0$. It is well-documented in the literature that the $solution$ has to be interpreted in the weak sense
  and one needs an entropy formulation to prove wellposedness. We refer to 
\cite{boris_2010,carrillo_1999,chenkarlsen_2005,chenperthame_2003,mostafakarlsen_2004,ValletAdvMathSciAppl} and references therein for more on entropy solution theory
for deterministic degenerate parabolic-hyperbolic equations. 

\subsection{Studies on degenerate parabolic-hyperbolic equations with Brownian noise}
  The study of stochastic degenerate parabolic-hyperbolic equations  has so far been limited to mainly equations with Brownian noise. In particular, hyperbolic conservation laws with Brownian noise are the examples of such problems that have attracted the attention of many. 
The first documented development in this direction is \cite{risebroholden1997}, where the authors established 
existence of path-wise weak solution (possibly non-unique) of one dimensional balance laws \textit{via} splitting method. 
In a separate development, Khanin \textit{et al.}
\cite{Sinai1997} published their celebrated work that described some statistical properties of Burgers equations with noise. 
J.~ U.~ Kim \cite{Kim_2003} extended Kruzkov's entropy formulation and established the wellposedness for one dimensional balance
laws that are driven by additive Brownian noise. Multidimensional case was studied by Vallet and Wittbold \cite{ValletWittbold}, and they
established wellposedness of entropy solution with the theory of Young-measures but in a bounded domain.

This approach is not applicable for multiplicative noise case. This was studied by many
authors (\cite{BaVaWit_2012,Chen-karlsen_2012,Vovelle2010,nualart:2008}).
In \cite{nualart:2008}, Feng and Nualart came up with a way to recover the necessary 
information in the form of  {\it strong entropy} condition from the parabolic regularisation and 
established the uniqueness of strong entropy solution in $L^p$-framework for several space dimensions but  the existence was for 
 one space dimension. We also add here that Feng and Nualart \cite{nualart:2008} uses an entropy formulation
which is strong in time but weak in space, which in our view may give rise to problems where the solutions are not shown to have
continuous sample paths.  We refer to \cite{BisMaj}, where a few technical questions are raised on the strong in time
formulation and remedial measures have been proposed. In \cite{Vovelle2010}, the authors obtain the existence \textit{via} kinetic
formulation and  \cite{Chen-karlsen_2012} uses BV solution framework. In a recent paper, Vallet \textit{et al.} 
\cite{BaVaWit_2012}, established the wellposedness \textit{via} the Young's measure approach. The wellposedness result of the multidimensional degenerate parabolic-hyperbolic stochastic problem 
has been studied by Vovelle, 
Hofmanova
 and Debussche \cite{Vovelledebussche_2014}, and Vallet \textit{et al.} \cite {BaVaWit_2014}.
In \cite{Vovelledebussche_2014}, they adapt the notion of kinetic formulation and develop a wellposedness theory. In \cite {BaVaWit_2014},
the authors revisited \cite{boris_2010,carrillo_1999,chenkarlsen_2005} and established the wellposedness of the entropy solution
\textit{via} Young's measure theory.

\subsection{Relevant studies on problems  with L\'{e}vy noise}
  Over the last decade there has been many contributions on the larger area of stochastic partial differential equations that
  are driven by L\'{e}vy noise. An worthy reference on this subject is \cite{peszat}. However, very little is available on the
  specific problem of degenerate parabolic problems with L\'{e}vy noise such as \eqref{eq:levy_stochconservation_laws}. This article marks an important step in our quest to develop a
  comprehensive theory of stochastic degenerate parabolic equations  that are driven by jump-diffusions. The relevant results in this context are made available recently and they are on conservation laws that are perturbed by L\'{e}vy noise. In  recent articles \cite{BisMajKarl_2014,BisKoleyMaj}, Biswas \textit{et al.} established existence, uniqueness of entropy solution 
 for multidimensional  conservation laws with Poisson noise {\it via} Young measure approach. In \cite{BisKoleyMaj}, the authors developed a continuous 
 dependence theory on nonlinearities within $BV$ solution setting. 
  
Stochastic degenerate parabolic-hyperbolic equations are one of the most important classes of nonlinear stochastic PDEs. Nonlinearity and degeneracy are two main features of
these equations and yield several striking phenomena. Therefore, it requires new mathematical ideas, approaches, and theories.
It is well-known that due to presence of nonlinear flux term,
solutions to \eqref{eq:levy_stochconservation_laws} are not smooth even for smooth initial data $u_0(x)$. Therefore the solutions 
must be interpreted in the weak sense. Before introducing the concept of weak solutions, we first recall the notion of predictable 
 $\sigma$-field. By a predictable $\sigma$-field on $[0,T]\times \Omega$, denoted
by $\mathcal{P}_T$, we mean that the $\sigma$-field generated by the sets of the form: $\{0\}\times A$ and $(s,t]\times B$ for any $A \in \mathcal{F}_0; B
\in \mathcal{F}_s,\,\, 0<s,t \le T$. 
The notion of stochastic weak solution is defined as follows.
\begin{defi}[Stochastic weak solution]\label{defi:weaksolution}
 An $ L^2(\R^d )$-valued $\{\mathcal{F}_t: t\geq 0 \}$-predictable stochastic process $u(t)= u(t,x)$ is said to be a weak solution
  to our problem \eqref{eq:levy_stochconservation_laws} provided 
  \begin{itemize}
   \item [1)]  $u \in L^2(\Omega \times \Pi_T)$ and $\phi(u)\in L^2((0,T)\times \Omega; H^1(\R^d))$.
  \item[(2)] $\frac{\partial}{\partial t} [u-\int_0^t \int_{E} \eta(x,u(s,\cdot);z) \tilde{N}(dz,ds)]
  \in L^2((0,T)\times \Omega; H^{-1}(\R^d))$
  in the sense of distribution.
  \item[(3)] For almost every $t\in [0,T]$ and $ P-$ a.s, the following variational formulation holds: 
  \begin{align}
   0=&\Big\langle \frac{\partial}{\partial t} [u-\int_0^t \int_{E} \eta(x, u(s,\cdot);z) \tilde{N}(dz,ds)],v \Big\rangle_{H^{-1}(\R^d),H^1(\R^d)} \notag \\
   & \qquad \quad + \int_{\R^d} \Big\{ \grad \phi(u(t,x)) + f(u(t,x))\Big\}.\grad v\,dx,
  \end{align}
   for any $v \in H^1(\R^d)$.
  \end{itemize}
\end{defi}
However, it is well-known that weak solutions may be discontinuous and they are not uniquely determined by their initial data.
Consequently, an admissibility  criterion for so called {\em entropy solution}~(see Section \ref{technical} for the definition of 
entropy solution) must be imposed to single out the physically correct solution.
  
\subsection{Goal of the study and outline of the paper} The case of a strongly degenerate stochastic problem driven by Brownian 
  noise is studied by {Bauzet \textit{et al.} \cite{BaVaWit_2014}. 
  In this article, drawing primary motivation from \cite{BaVaWit_2014,BisMajKarl_2014,carrillo_1999}, we propose to establish the 
  wellposedness of the entropy solution to degenerate Cauchy problem \eqref{eq:levy_stochconservation_laws} by using vanishing 
  viscosity method along with few \textit{a priori} bounds. 
  \vspace{.1cm}
  
  The rest of the paper is organized as follows. We state the assumptions, details of the technical framework and state the main 
  results in Section \ref{technical}. Section \ref{sec:existence-weak-solu-viscous} is devoted to prove the existence of weak solution for 
  viscous problem \textit{via} implicit time discretization scheme and to derive some \textit{a priori} estimates for viscous solution. In section \ref{sec:existence-entropy},
  we first establish uniqueness of the limit of viscous solutions as viscous parameter goes to zero \textit{via} Young measure theory and then we 
  establish existence of entropy solution. The uniqueness of the entropy solution is presented 
  in the final section.

\section{Technical framework and statements of the main results}\label{technical}
Here and in the sequel, 
we denote by $N^2_\omega(0,T,L^2(\R^d))$ the space of predictable $L^2(\R^d)$-valued processes $u$ such that $\E\Big[\int_{\Pi_T}|u|^2\,dt\,dx\Big]<+\infty$. Moreover, we use
the letter $C$ to denote various generic constants. There are situations where constants may change 
from line to line, but the notation is kept unchanged so long as it does not impact the primary implication. We denote $c_\phi$ and $c_f$ the Lipschitz
constants of $\phi$, and $f$ respectively. Also, we use $\big\langle ,\big\rangle$ to denote the pairing between $H^1(\R^d)$ and 
$H^{-1}(\R^d)$.

\subsection{Entropy inequalities}
 We begin this subsection with a formal derivation of entropy inequalities \`{a} la Kruzkov. Remember that we need to replace 
 the traditional chain rule for deterministic calculus by It\^{o}-L\'{e}vy chain rule.
 \begin{defi}[Entropy flux triple]
A triplet $(\beta,\zeta,\nu) $ is called an entropy flux triple if $\beta \in C^2(\R)$, Lipschitz and $\beta \ge0$,
$\zeta = (\zeta_1,\zeta_2,....\zeta_d):\R \mapsto \R^d$ is a vector valued function, and $ \nu :\R \mapsto \R $ is a scalar valued
function such that 
\[\zeta'(r) = \beta'(r)f'(r) \quad \text{and}\quad \nu^\prime(r)= \beta'(r)\phi'(r).\]
 An entropy flux triple $(\beta,\zeta,\nu)$ is called convex if $ \beta^{\prime\prime}(s) \ge 0$.  
\end{defi}
For a small positive number $\eps>0$, assume that the parabolic perturbation
\begin{align}
  du(t,x) -\Delta \phi(u(t,x))\,dt =& \mbox{div}_x f(u(t,x)) \,dt +\int_{E} \eta(x, u(t,x);z)\tilde{N}(dz,dt) \notag \\
 &  \hspace{3cm}+ \eps \Delta u(t,x)\,dt, \quad  (t,x)\in \Pi_T \label{eq:levy_laws-viscous}
\end{align}
of \eqref{eq:levy_stochconservation_laws} has a unique weak solution  $u_\eps(t,x)$.  Note that this weak solution $u_\eps\in  L^2((0,T)\times \Omega; H^1(\R^d))$.  Moreover, for the time being, we assume that it satisfies the initial condition in the sense of \eqref{weak-initial-consition}.
This enables one to derive a weak version of It\^{o} -L\'{e}vy formula for the 
solutions to 
  \eqref{eq:levy_stochconservation_laws}, as detailed in the Theorem \ref{thm:weak-its} in the Appendix. 
\vspace{.2cm}

 Let $(\beta,\zeta,\nu)$ be an entropy flux triple. Given a nonnegative test function $\psi\in C_{c}^{1,2}([0,\infty)\times
 \R^d)$, we apply generalised version of the It\^{o}-L\'{e}vy formula to have, for almost every $T >0$, 
\begin{align}
  &\int_{\R^d} \beta(u_\eps(T,x))\psi(T,x)\,dx -  \int_{\R^d} \beta(u_\eps(0,x))\psi(0,x)\,dx  \notag \\
  = & \int_{\Pi_T} \beta(u_\eps(t,x)) \partial_t\psi(t,x) \,dx\,dt -  \int_{\Pi_T} \grad \psi(t,x)\cdot \zeta(u_\eps(t,x))\,dx\,dt 
   \notag \\
 + & \int_{\Pi_T} \int_{E}  \int_0^1 \eta(x,u_\eps(t,x);z)\beta^\prime (u_\eps(t,x)
 + \theta\,\eta(x,u_\eps(t,x);z))\psi(t,x)\,d\theta\,\tilde{N}(dz,dt)\,dx \notag \\
+&\int_{\Pi_T}\int_{E}  \int_0^1  (1-\theta)\eta^2(x,u_\eps(t,x);z)\beta^{\prime\prime} (u_\eps(t,x) + \theta\,\eta(x,u_\eps(t,x);z))
\psi(t,x)\,d\theta\,m(dz)\,dx\,dt \notag \\
 -&   \int_{\Pi_T} \Big(\eps \nabla_x \psi(t,x).\nabla_x\beta(u_\eps(t,x)) +\eps \beta''(u_\eps(t,x))|\nabla_x u_\eps(t,x)|^2\psi(t,x)\Big)\,dx\,dt \notag \\
-&  \int_{\Pi_T} \phi^\prime(u_\eps(t,x)) \beta^{\prime\prime}(u_\eps(t,x)) \big|\grad u_\eps(t,x)\big|^2\psi(t,x)\,dx\,dt
+  \int_{\Pi_T} \nu(u_\eps(t,x))\Delta \psi(t,x)\,dx\,dt
\end{align}
 Let $G$ be the associated Kirchoff's function of $\phi$, given by $G(x)= \int_0^x \sqrt{\phi^\prime(r)}\,dr$.
 A simple calculation shows that $|\Grad G(u_\eps(t,x))|^2=\phi^\prime(u_\eps(t,x)) |\grad u_\eps(t,x)|^2 $. Since $\beta$ and $\psi$
 are nonnegative functions, we obtain
 \begin{align}
  0\le  & \int_{\R^d} \beta(u_\eps(0,x))\psi(0,x)\,dx  +  \int_{\Pi_T} \Big\{ \beta(u_\eps(t,x)) \partial_t\psi(t,x)
    -  \grad \psi(t,x)\cdot \zeta(u_\eps(t,x)) \Big\}dx\,dt   \notag \\
-&  \int_{\Pi_T} \beta^{\prime\prime}(u_\eps(t,x)) |\Grad G(u_\eps(t,x))|^2\psi(t,x)\,dx\,dt
+  \int_{\Pi_T} \nu(u_\eps(t,x))\Delta \psi(t,x)\,dx\,dt + \mathcal{O}(\eps) \notag \\
 + & \int_{\Pi_T}\int_{E}  \int_0^1 \eta(x,u_\eps(t,x);z)\beta^\prime \big(u_\eps(t,x) + \theta\,\eta(x,u_\eps(t,x);z)\big)\psi(t,x)\,d\theta\,\tilde{N}(dz,dt)\,dx \notag \\
+&\int_{\Pi_T} \int_{E} \int_0^1  (1-\theta)\eta^2(x,u_\eps(t,x);z)\beta^{\prime\prime} \big(u_\eps(t,x) + \theta\,\eta(x,u_\eps(t,x);z)\big)
\psi(t,x)\,d\theta\,m(dz)\,dx\,dt \notag
\end{align}

Clearly, the above inequality is stable under the limit $\eps \goto 0 $, if the family $\{u_\eps \}_{\eps >0}$ has $ L_{\mathrm{loc}}^p $-type stability.
Just as the deterministic equations, the above inequality provides us with the entropy condition.
We now formally define the entropy solution. 

\begin{defi} [Stochastic entropy solution]\label{defi:stochentropsol}
A stochastic process $u \in N^2_\omega(0,T,L^2(\R^d))$ is called a stochastic entropy solution of \eqref{eq:levy_stochconservation_laws} if
\begin{itemize}
 \item[(1)] for each $ T>0$, $G(u) \in L^2((0,T)\times \Omega;H^1(\R^d))$ and 
 $\underset{0\leq t\leq T}\sup  \E\big[||u(t)||_{2}^{2}\big] < \infty$. 
\item[(2)] Given a nonnegative test function  $\psi\in C_{c}^{1,2}([0,\infty )\times\R^d) $ and a convex entropy flux triple
$(\beta,\zeta,\nu)$, the following inequality holds:
\begin{align}
&  \int_{\Pi_T} \Big\{ \beta(u(t,x)) \partial_t\psi(t,x)
+  \nu(u(t,x))\Delta \psi(t,x) -  \grad \psi(t,x)\cdot \zeta(u(t,x)) \Big\}dx\,dt \notag \\
 & + \int_{\Pi_T} \int_{E} \int_0^1 \eta(x,u(t,x);z)\beta^\prime (u(t,x) + \theta\,\eta(x,u(t,x);z))\psi(t,x)\,d\theta\,\tilde{N}(dz,dt)\,dx \notag \\
&\quad  +\int_{\Pi_T} \int_{E}  \int_0^1  (1-\theta)\eta^2(x,u(t,x);z)\beta^{\prime\prime} (u(t,x) + \theta\,\eta(x,u(t,x);z))
\psi(t,x)\,d\theta\,m(dz)\,dx\,dt \notag \\
& \qquad \quad \ge  \int_{\Pi_T} \beta^{\prime\prime}(u(t,x)) |\Grad G(u(t,x))|^2\psi(t,x)\,dx\,dt
- \int_{\R^d} \beta(u_0(x))\psi(0,x)\,dx, \quad P-\text{a.s}.\label{eq:entropy-def}
\end{align}
\end{itemize}
\end{defi} 

\begin{rem} We point out that, by a classical separability argument, it is possible to choose a subset of $\Omega$ of $P$-full measure such that \eqref{eq:entropy-def} holds on that subset for every admissible entropy triplet and test function. 
\end{rem}

The primary aim of this paper is to establish the existence and uniqueness of entropy solutions for the  Cauchy problem 
\eqref{eq:levy_stochconservation_laws} in accordance with Definition \ref{defi:stochentropsol}, and we do so under
the following assumptions:
 \begin{Assumptions}
  \item \label{A1}  $ \phi:\R\rightarrow \R$ is a non-decreasing  Lipschitz continuous function with $\phi(0)=0$.
 Moreover, if $\eta$ is not a constant function with respect to the space variable $x$, $t\longmapsto \sqrt{\phi^\prime(t)}$ has a modulus of continuity $\omega_\phi$
 such that $\frac{\omega_\phi(r)}{r^{\frac{2}{3}}} \longrightarrow 0$ as $r \goto 0$. 
\item \label{A2}  $ f=(f_1,f_2,\cdots, f_d):\R\rightarrow \R^d$ is  a Lipschitz continuous function with $f_k(0)=0$ for 
all $1\le k\le d$.
 \item \label{A4} The space $E$ is of the form $\mathcal{O}\times \R^*$ and the Borel measure $m$ on $E$ has the form $\lambda\times \mu $ where 
 $\lambda$ is a Radon measure on $\mathcal{O}$ and $\mu$ is a so-called one dimensional L\'{e}vy measure.

\item \label{A3} There exist positive constants  $K > 0$, $\lambda^* \in (0,1)$ and $h_1(z) \in L^2(E,m)$ with $0\le h_1(z)\le 1$  such that 
 \begin{align*} \big| \eta(x,u;z)-\eta(y,v;z)\big|  \leq  (\lambda^* |u-v|+ K |x-y|) h_1(z) ~\text{for all}~ x,y \in \R^d;~ u,v \in \R;~~z\in E.
 \end{align*}
\item  \label{A5} There exists a nonnegative function $ g\in L^\infty(\R^d)\cap L^2(\R^d)$ and $h_2(z)\in L^2(E,m)$ such that for all $(x,u,z)\in \R^d \times \R\times E,$
\begin{align*}
|\eta(x,u;z)| \le g(x)(1+|u|)h_2(z).
\end{align*}    
\end{Assumptions}

 The above definition does not say anything explicitly about the entropy solution satisfying the initial condition. However, 
 the initial condition is satisfied in a certain weak sense. Here we state the lemma whose proof follows a simple line argument as in the 
 Lemma $2.3$ of \cite{BisMajKarl_2014}.
\begin{lem}\label{lem:initial-cond}
Any  entropy solution $u(t,\cdot)$ of 
\eqref{eq:levy_stochconservation_laws} satisfies the initial condition in the following sense: for every non negative test
function $\psi\in C_c^2(\R^d)$ such that $\supp(\psi) = K$
\begin{align}
  \lim_{h\rightarrow 0}\E \Big[\frac 1h \int_0^h\int_K  \big|u(t,x) -u_0(x)\big|\psi(x)\,dx\, dt \Big]= 0.  \label{eq:intial_cond_weak}
\end{align}
\end{lem}

Next, we describe a special class of entropy functions that plays an important role in later analysis. 
  Let $\beta:\R \rightarrow \R$ be a $C^\infty$ and  Lipschitz function satisfying 
 \begin{align*}
      \beta(0) = 0,\quad \beta(-r)= \beta(r),\quad \beta^{\prime\prime} \ge 0,
 \end{align*} and 
\begin{align*}
\beta^\prime(r)=\begin{cases} -1\quad \text{when} ~ r\le -1,\\
                               \in [-1,1] \quad\text{when}~ |r|<1,\\
                               +1 \quad \text{when} ~ r\ge 1.
                 \end{cases}
\end{align*} For any $\vartheta > 0$, define  $\beta_\vartheta:\R \rightarrow \R$ by 
\begin{align*}
         \beta_\vartheta(r) = \vartheta \beta(\frac{r}{\vartheta}).
\end{align*} Then
\begin{align}\label{eq:approx to abosx}
 |r|-M_1\vartheta \le \beta_\vartheta(r) \le |r|\quad \text{and} \quad |\beta_\vartheta^{\prime\prime}(r)| \le \frac{M_2}{\vartheta} {\bf 1}_{|r|\le \vartheta},
\end{align} where
\begin{align*}
 M_1 = \sup_{|r|\le 1}\big | |r|-\beta(r)\big |, \quad M_2 = \sup_{|r|\le 1}|\beta^{\prime\prime} (r)|.
\end{align*}
By simply dropping $\vartheta$, for $\beta= \beta_\vartheta$ ~ we define  

\begin{equation*}
\begin{cases}
 \phi^\beta(a,b)=\int_{b}^a \beta^\prime(\sigma-b)\phi^\prime(\sigma)\,d(\sigma),\quad 
 F_k^\beta(a,b)=\int_{b}^a \beta^\prime(\sigma-b)f_k^\prime(\sigma)\,d(\sigma),\\
 F_k(a,b)= \text{sign}(a-b)(f_k(a)-f_k(b)),\quad  F(a,b)= \big(F_1(a,b),F_2(a,b),....,F_d(a,b)\big).
 \end{cases}
\end{equation*}
\vspace{.1cm}

We conclude this section by stating  the main results of this paper. 
\begin{thm}(Existence)\label{thm:existenc}
 Let the assumptions \ref{A1}-\ref{A5} be true, and that $ L^2(\R^d)$-valued $\mathcal{F}_0$-measurable random variable 
 $ u_0$ satisfies $\E\big[||u_0||^2_2\big] < \infty $.
 Then, there exists an entropy solution of \eqref{eq:levy_stochconservation_laws} in the sense of Definition
 \ref{defi:stochentropsol}.
\end{thm}

\begin{thm}(Uniqueness)\label{thm:uniqueness}
 Let the assumptions \ref{A1}-\ref{A5} be true, and that $ L^2(\R^d)$-valued  $\mathcal{F}_0$-measurable random variable
 $ u_0$ satisfies $\E\big[||u_0||^2_2\big] < \infty $.
 Then, the entropy solution of \eqref{eq:levy_stochconservation_laws} is unique.
\end{thm}

\begin{rem} \label{lem:p-bounds}
In addition, if   $u_0$ is $L^p(\R^d)$ for $p\in [2,\infty)$ then it could be concluded that $u \in L^\infty(0,T,L^p(\Omega\times \R^d))$. Furthermore, if $u_0\in L^\infty$ and there is $M>0$ such that $\eta(x,u;z)=0$ for $|u|>M$ and $M_1=\sup_{x, |u|\le M, z} | \eta(x, u; z)| <\infty$, then $|u(t,x)|\le \max\{M+ M_1, ||u_0||_{\infty}\}$ for almost every $(t,x,\omega)\in \Pi_T\times \Omega$.  We sketch a justification of this claim in Section 4. 
\end{rem}
\section{Existence of weak solution for viscous problem}\label{sec:existence-weak-solu-viscous}
Just as the deterministic problem, here also we study the corresponding regularized problem by adding a small diffusion operator
and derive some \textit{a priori} bounds. Due to the nonlinear function $\phi$ and related degeneracy,
one cannot expect classical solution and instead seeks an weak solution.

\subsection{Existence of weak solution to viscous problem}
 For a small parameter $\eps >0$, we consider the viscous approximation of \eqref{eq:levy_stochconservation_laws} as
 \begin{align}
  du(t,x) -\Delta \phi(u(t,x))\,dt =& \mbox{div}_x f(u(t,x)) \,dt +\int_{E} \eta(x, u(t,x);z)\tilde{N}(dz,dt) \notag \\
 & \qquad \quad  + \eps \Delta u(t,x)\,dt,\quad t>0, ~ x\in \R^d.\label{eq:levy_stochconservation_laws-viscous}
\end{align}
In this subsection, we establish the existence of a weak solution for the problem \eqref{eq:levy_stochconservation_laws-viscous}.
  To do this, we use  an implicit time discretization scheme. Let $\Delta t= \frac{T}{N}$ for some positive integer $N \ge 1$.
  Set $t_n= n\,\Delta t$ for $n=0,1,2\,\cdots, N$.
  
 Define
 \begin{align}
  \mathcal{N}= L^2(\Omega;H^1(\R^d)),\quad
  \mathcal{N}_n= \{ \text{the }\mathcal{F}_{n\Delta t}\text{ measurable elements of }\mathcal{N}\},\notag \\
  \mathcal{H}= L^2(\Omega;L^2(\R^d)),\quad
  \mathcal{H}_n= \{ \text{the }\mathcal{F}_{n\Delta t}\text{ measurable elements of }\mathcal{H}\}.\notag 
 \end{align}
 \begin{prop} \label{prop:time discretization}
  Assume that $\Delta t$ is small. For any given $u_n \in \mathcal{H}_n$, there exists a
   unique
   $u_{n+1} \in \mathcal{N}_{n+1}$ with $\phi(u_{n+1}) \in \mathcal{N}_{n+1}$ such that $P-\text{a.s.}$ for any $v \in H^1(\R^d)$, the following variational formula 
   holds:
   \begin{align}
    &\int_{\R^d} \Big((u_{n+1}-u_n)v + \Delta t \Big\{ \grad \phi(u_{n+1}) + \eps \grad u_{n+1} + f(u_{n+1})\Big\}\cdot \grad v \Big)\,dx \notag \\
    &= \int_{\R^d} \int_{t_n}^{t_{n+1}} \int_{E} \eta(x,u_n;z)\,v\, \tilde{N}(dz,ds)\,dx. \label{variational_formula_discrete}
   \end{align}
 \end{prop}

 Before proving the proposition, first we state a key deterministic lemma, related to the weak solution of degenerate parabolic
   equations. We have the following lemma, a proof of which could be found in [ page $19$, \cite{brezis} ].
\begin{lem} \label{lem:deterministic}
 Assume that $\Delta t$ is small and $X \in L^2(\R^d)$. Then, for fixed positive parameter $\eps>0$,
 \begin{itemize}
  \item[(1)] there exists a unique $u \in H^1(\R^d)$ with $\phi(u) \in H^1(\R^d)$ such that, for any $v\in H^1(\R^d)$
 \begin{align}
  \int_{\R^d} \Big(uv + \Delta t \Big\{ \grad \phi(u) + \eps \grad u + f(u)\Big\}\cdot \grad v \Big)\,dx
    = \int_{\R^d} Xv\,dx.\label{variational_formula}
 \end{align}
 \item[(2)] There exists a constant $C= C(\Delta t)>0$ such that the following  \textit{a priori}
 estimate holds
 \begin{align}
  ||u||_{L^2(\R^d)}^2 + ||\phi(u)||_{H^1(\R^d)}^2 + \eps ||\grad u||_{L^2(\R^d)}^2 \le C ||X||_{L^2(\R^d)}^2.
  \label{estimate:a-priori-deterministic} 
 \end{align}
  \item[(3)] The map $\Theta : X \in L^2(\R^d) \mapsto (u,\phi(u)) \in H^1(\R^d)^2$ is continuous.
 \end{itemize}
\end{lem}
\vspace{.2cm}

\noindent{\bf Proof of the Proposition \ref{prop:time discretization}} 
Let $u_n \in \mathcal{N}_n$.
Take $X= u_n +\int_{t_n}^{t_{n+1}} \int_{E} \eta(x,u_n;z)\tilde{N}(dz,ds) $. Then, by the assumption \ref{A5}, we
obtain 
\begin{align*}
 \E \Big[||X||^2_{L^2(\R^d)}\Big]\le ||u_n||_{\mathcal{H}}^2 + C\,\Delta t \big(||g||_{L^2(\R^d)}^2 + ||u_n||_{\mathcal{H}}^2 \big).
\end{align*}
This shows that for a.s. $\omega \in \Omega$, $X \in  L^2(\R^d)$. Therefore, one can use the Lemma \ref{lem:deterministic}, and conclude 
that for almost surely $\omega \in \Omega$, there exist unique $u(\omega)$ satisfying  the variational equality \eqref{variational_formula_discrete}. 
Moreover, by construction $X \in \mathcal{H}_{n+1}$. Thus, due to the continuity of $\Theta$ for the $\mathcal{F}_{(n+1)\Delta t}$
measurability and to \textit{a priori} estimate \eqref{estimate:a-priori-deterministic}, we conclude that $u \in \mathcal{N}_{n+1}$ with $\phi(u) \in \mathcal{N}_{n+1}$.
 We denote this solution $u$ by $u_{n+1}$. Hence the proof of the proposition follows.
 
 \subsubsection{\bf \textit{A priori} estimate}

 Note that, for any $v\in \mathcal{D}(\R^d)$, $\int_{\R^d} f(v)\cdot \grad v\,dx=0$ and hence true for any $v\in H^1(\R^d)$ by density argument.
 We choose a test function $v= u_{n+1}$ in \eqref{variational_formula_discrete} and have
\begin{align*}
 &\int_{\R^d} (u_{n+1}-u_n)u_{n+1}\,dx + \Delta t \int_{\R^d} \phi^\prime(u_{n+1})|\grad u_{n+1}|^2\,dx
 + \eps\,\Delta t \int_{\R^d}|\grad u_{n+1}|^2\,dx\notag \\
 & = \int_{\R^d} \int_{t_n}^{t_{n+1}} \int_{E} \eta(x,u_n;z) \, \tilde{N}(dz,ds)u_{n+1}\,dx \notag \\
 & \le \int_{\R^d} \int_{t_n}^{t_{n+1}} \int_{E} \eta(x,u_n;z)u_n\, \tilde{N}(dz,ds)\,dx + \frac{\alpha}{2} ||u_{n+1}-u_n||_{L^2(\R^d)}^2 \notag \\
  & \qquad + \frac{1}{2\alpha} \int_{\R^d} \Big(\int_{t_n}^{t_{n+1}} \int_{E} \eta(x,u_n;z)\, \tilde{N}(dz,ds)\Big)^2\,dx,\quad
  \text{for some}~~\alpha >0.
\end{align*}
Since $\int_{\R^d} |\grad \phi(u)|^2\,dx = \int_{\R^d} |\phi^\prime(u)\,\grad u|^2\,dx \le c_\phi \int_{\R^d} \phi^\prime(u)
|\grad u|^2\,dx$, we see that 
 \begin{align}
 \frac{\Delta t}{c_\phi} ||\grad \phi(u)||_{\mathcal{H}}^2 \le \Delta t \, \E \Big[\int_{\R^d} \phi^\prime(u)|\grad u|^2\,dx\Big].\label{esti:grad-phi}
 \end{align}
 In view of the assumption \ref{A5}, the inequality \eqref{esti:grad-phi}, It\^{o}-L\'{e}vy isometry, and the fact that 
 for any $a,b \in \R, (a-b)a
 = \frac{1}{2}(a^2 + (a-b)^2 -b^2)$, we obtain
\begin{align*}
 &\frac{1}{2} \Big[ ||u_{n+1}||_{\mathcal{H}}^2 + ||u_{n+1}-u_n||_{\mathcal{H}}^2 - ||u_n||_{\mathcal{H}}^2 \Big] +
  \frac{\Delta t}{c_\phi} ||\grad \phi(u_{n+1})||_{\mathcal{H}}^2 +  \eps \,\Delta t ||\grad u_{n+1}||_{\mathcal{H}}^2 \notag \\
  & \le  \frac{\alpha}{2} ||u_{n+1}-u_n||_{\mathcal{H}}^2 + \frac{C\,\Delta t}{2\alpha} \big( 1+ ||u_n||_{\mathcal{H}}^2\big). 
\end{align*}
 Since $\alpha >0$ is arbitrary, one can choose $\alpha >0$ so that
\begin{align}
 & ||u_{n}||_{\mathcal{H}}^2 +  \sum_{k=0}^{n-1}||u_{k+1}-u_k||_{\mathcal{H}}^2 +
  \frac{\Delta t}{c_\phi}\sum_{k=0}^{n-1} ||\grad \phi(u_{k+1})||_{\mathcal{H}}^2 +  \eps \Delta t \sum_{k=0}^{n-1}
  ||\grad u_{k+1}||_{\mathcal{H}}^2 \notag \\
  & \le  C_1 + C_2\Delta t \sum_{k=0}^{n-1} ||u_{k}||_{\mathcal{H}}^2,\quad \text{for some constants}~~ C_1, C_2 >0.\label{esti:discrete-1}
\end{align}
Thanks to discrete Gronwall's lemma, one has from \eqref{esti:discrete-1},
\begin{align}
  ||u_{n}||_{\mathcal{H}}^2 +  \sum_{k=0}^{n-1}||u_{k+1}-u_k||_{\mathcal{H}}^2 +
  \frac{\Delta t}{c_\phi}\sum_{k=0}^{n-1} ||\grad \phi(u_{k+1})||_{\mathcal{H}}^2 +  \eps \Delta t \sum_{k=0}^{n-1}
  ||\grad u_{k+1}||_{\mathcal{H}}^2 \le  C \label{a-prioriestimate:1}
\end{align}
For fixed $\Delta t = \frac{T}{N}$, we define 
\begin{equation*}
 u^{\Delta t} (t)= \sum_{k=1}^N u_k {\bf 1}_{[(k-1)\Delta t, k \Delta t )}(t); \quad 
  \tilde{u}^{\Delta t}(t)= \sum_{k=1}^N \Big[ \frac{u_k - u_{k-1}}{\Delta t}( t- (k-1)\Delta t) + u_{k-1}\Big] 
  {\bf 1}_{[(k-1)\Delta t, k \Delta t )}(t)
 \end{equation*} 
 with $u^{\Delta t} (t)= u_0$ for $t<0$. Similarly, we define 
 \begin{align*}
 & \tilde{B}^{\Delta t}(t)= \sum_{k=1}^N \Big[ \frac{B_k - B_{k-1}}{\Delta t}[ t- (k-1)\Delta t] + B_{k-1}\Big] {\bf 1}_{[(k-1) 
 \Delta t, k \Delta t )}(t), 
\end{align*}
where 
\begin{align*}
 B_n &= \sum_{k=0}^{n-1} \int_{k\Delta t} ^{(k+1)\Delta t} \int_{E} \eta(x,u_k;z) \tilde{N}(dz,ds)
 = \int_0^{n \Delta t} \int_{E} \eta(x,u^{\Delta t}(s-\Delta t);z) \tilde{N}(dz,ds).
\end{align*}
A straightforward calculation shows that 
\begin{equation*}
 \begin{cases}
 \big\|u^{\Delta t}\big \|_{L^\infty(0,T;\mathcal{H})}= \underset{k=1,2,\cdots,N}\max\, \big\|u_k\big\|_{\mathcal{H}}; \quad
   \big\|\tilde{u}^{\Delta t}\big \|_{L^\infty(0,T;\mathcal{H})}= \underset{k=0,1,\cdots,N}\max\, \big\|u_k\big \|_{\mathcal{H}}, \\
   \big \| u^{\Delta t}-\tilde{u}^{\Delta t}\big\|_{L^2(0,T;\mathcal{H})}^2 \le \Delta t \sum_{k=0}^{N-1} \big \|u_{k+1}-u_k\big\|_{\mathcal{H}}^2
  .
  \end{cases}
\end{equation*}
\vspace{.1cm}

Since $\phi$ is a Lipschitz continuous function with $\phi(0)=0$, in view of the above definitions and \textit{a priori} estimate 
\eqref{a-prioriestimate:1}, we have the following proposition.
\begin{prop} \label{Prop: a-priori_bound_1}
Assume  that $\Delta t$ is small. Then  $u^{\Delta t},\, \tilde{u}^{\Delta t}$ are bounded sequences in $L^\infty(0,T;\mathcal{H})$;
 $\phi(u^{\Delta t})$, $\sqrt{\epsilon}u^{\Delta t}$ are a bounded sequences in $L^2(0,T;\mathcal{N})$ 
and $ || u^{\Delta t}-\tilde{u}^{\Delta t}||_{L^2(0,T;\mathcal{H})}^2
 \le C \Delta t$. \\ Moreover, $u^{\Delta t}-u^{\Delta t}(\cdot -\Delta t)\goto 0$ in $L^2 (\Omega \times \Pi_T)$.
\end{prop}

Next, we want to find some upper bound for  $\tilde{B}^{\Delta t}(t)$. Regarding this, we have the following proposition.
\begin{prop}\label{Prop: a-priori_bound_2}
$\tilde{B}^{\Delta t}$ is a bounded sequence in $L^2(\Omega \times \Pi_T)$ and 
\begin{align}
 \Big\| \tilde{B}^{\Delta t}(\cdot) - \int_{0}^\cdot \int_{E}  \eta(x,u^{\Delta t}(s-\Delta t);z) 
 \tilde{N}(dz,ds)\Big\|_{L^2(\Omega \times \R^d)}^2 \le C \Delta t. \notag
\end{align}
\end{prop}
\begin{proof} First we prove the boundedness of $\tilde{B}^{\Delta t}(t)$. By using definition of $\tilde{B}^{\Delta t}(t)$,
the assumption \ref{A5}, and boundedness of $u^{\Delta t}$ in  $L^\infty(0,T;\mathcal{H})$,  we obtain
\begin{align*}
 \big\| \tilde{B}^{\Delta t} \big\|_{L^2\big(0,T;L^2(\Omega,L^2(\R^d))\big)}^2 &\le \Delta t \sum_{k=0}^N ||B_k||_{L^2(\Omega \times \R^d)}^2 \notag \\
 & \le   \Delta t \sum_{k=0}^N  \E \Big[\Big| \int_{\R^d} \int_{0}^{k\Delta t}\int_{E}  \eta(x,u^{\Delta t}(s-\Delta t);z) \tilde{N}(dz,ds)\,dx
 \Big|^2\Big] \notag \\
 & \le C \Delta t \sum_{k=0}^N  \E \Big[ \int_{\R^d} \int_{0}^{k\Delta t} g^2(x)\big(1+ |u^{\Delta t}(s-\Delta t)|^2\big)\,dx\,ds\Big] \notag \\
 & \le C\big(1+ ||u^{\Delta t}||_{L^\infty(0,T;L^2(\Omega \times \R^d))}\big) <  + \infty.
\end{align*}
Thus,  $\tilde{B}^{\Delta t}$ is a bounded sequence in $L^2(\Omega \times \Pi_T)$.
\vspace{.2cm}

To prove second part of the proposition, we see that for any $t\in \big[n\Delta t,(n+1)\Delta t\big)$,
\begin{align*}
  &\tilde{B}^{\Delta t}(t)-  \int_{0}^t \int_{E}  \eta(x,u^{\Delta t}(s-\Delta t);z) \tilde{N}(dz,ds) \notag \\
   = &\frac{t-n\Delta t}{\Delta t} \int_{n\Delta t}^{(n+1)\Delta t} \int_{E} \eta(x,u^{\Delta t}(s-\Delta t);z) \tilde{N}(dz,ds)
  - \int_{n\Delta t}^{ t} \int_{E} \eta(x,u^{\Delta t}(s-\Delta t);z) \tilde{N}(dz,ds) \notag \\
   =& \frac{t-n\Delta t}{\Delta t} \int_{n\Delta t}^{(n+1)\Delta t} \int_{E} \eta(x,u_n;z) \tilde{N}(dz,ds)
  - \int_{n\Delta t}^{ t} \int_{E} \eta(x,u_n;z) \tilde{N}(dz,ds).
\end{align*}
Therefore, in view of \eqref{a-prioriestimate:1} and assumption \ref{A5}, we have
\begin{align*}
 &  \Big\| \tilde{B}^{\Delta t}(t) - \int_{0}^t \int_{E}  \eta(x,u^{\Delta t}(s-\Delta t);z) \tilde{N}(dz,ds)\Big\|_{L^2(\Omega \times \R^d)}^2 \notag \\
 =& \E\Bigg[ \int_{\R^d} \Big| \frac{t-n\Delta t}{\Delta t} \int_{n\Delta t}^{(n+1)\Delta t} \int_{E} \eta(x,u_n;z) \tilde{N}(dz,ds)
  - \int_{n\Delta t}^{ t} \int_{E} \eta(x,u_n;z) \tilde{N}(dz,ds)\Big|^2\,dx \Bigg] \\
   \le & 2  \int_{\R^d} \E\Bigg[ \Big( \frac{t-n\Delta t}{\Delta t}\Big)^2  \int_{n\Delta t}^{(n+1)\Delta t} \int_{E} \eta^2(x,u_n;z)\,m(dz)\,ds 
    + \int_{n\Delta t}^{ t} \int_{E} \eta^2(x,u_n;z)\,m(dz)\,ds\Bigg]dx   \\
    \le & C\big(1+ ||u_n||_{\mathcal{H}}^2\big) \Big[ \frac{(t-n\Delta t)^2}{\Delta t} + (t-n\Delta t)\Big] \le  C\,\Delta t.
\end{align*} 
This completes the proof.
\end{proof}
\subsubsection{\bf Convergence of $u^{\Delta t}(t,x)$}

Thanks to Proposition \ref{Prop: a-priori_bound_1} and Lipschitz property of $f$ and $\phi$, there exist $u,~ \phi_u$ and $ f_u$ 
such that (up to a subsequence) 

\begin{equation}\label{convergence:weak-1}
\left\{\begin{array}{lcl}
 u^{\Delta t}\rightharpoonup^* u & \text{in} &  L^\infty\big(0,T;L^2(\Omega \times \R^d)\big)  \\
u^{\Delta t}\rightharpoonup u & \text{in}&  L^2\big((0,T)\times\Omega; H^1(\R^d)\big)  \quad\text{ (for fixed positive $\epsilon$)}
\\
  \phi(u^{\Delta t}) \rightharpoonup \phi_u & \text{in}&  L^2\big((0,T)\times \Omega;H^1(\R^d)\big) \\
  F(u^{\Delta t}) \rightharpoonup f_u & \text{in} &  L^2\big((0,T)\times\Omega\times\R^d\big).
  \end{array}
  \right.
\end{equation} 

Next, we want to identify the weak limits $\phi_u$ and $f_u$. 
Note that, for any $v\in H^1(\R^d)$, we can rewrite \eqref{variational_formula_discrete}, in terms of $u^{\Delta t}, \tilde{u}^{\Delta t}$
and $\tilde{B}^{\Delta t}$ as 
\begin{align}
 \int_{\R^d} \Big(\frac{\partial}{\partial t}(\tilde{u}^{\Delta t}-\tilde{B}^{\Delta t})(t)\, v + 
  \Big\{ \grad \phi(u^{\Delta t}(t)) + \eps \grad u^{\Delta t}(t) + f(u^{\Delta t}(t))\Big\}\cdot \grad v \Big)\,dx =0.
  \label{variational_formula_discrete_1}
\end{align}

\begin{lem}\label{lemma:cauchyness}
  $\{u^{\Delta t}\}$ is a Cauchy sequence in $L^2(\Omega \times \Pi_T)$.
\end{lem}
\begin{proof}
Consider two positive integers $N$ and $M$ and denote $\Delta t =\frac{T}{N}$, $\Delta s = \frac{T}{M}$. Then, for 
any $v\in H^1(\R^d)$, one gets from \eqref{variational_formula_discrete_1},
\begin{align}
& \int_{\R^d} \Bigg(\frac{\partial}{\partial t}\Big[(\tilde{u}^{\Delta t}-\tilde{B}^{\Delta t})(t) -(\tilde{u}^{\Delta s}-\tilde{B}^{\Delta s})(t)\Big]\, v + 
  \Big\{ \grad \Big(\phi(u^{\Delta t}(t)) -\phi(u^{\Delta s}(t))\Big) \notag \\
  &+ \eps \grad \Big(u^{\Delta t}(t)-u^{\Delta s}(t)\Big)+ \Big( f(u^{\Delta t}(t))- f(u^{\Delta s}(t))\Big)\Big\}\cdot \grad v
  \Bigg)\,dx =0.\label{variational_formula_discrete_2}
\end{align}
Let $w=(\tilde{u}^{\Delta t}-\tilde{B}^{\Delta t})(t) -(\tilde{u}^{\Delta s}-\tilde{B}^{\Delta s})(t) $. Set $v= (I-\Delta)^{-1}w$
in \eqref{variational_formula_discrete_2}. Then, one has
\begin{align}
 & \frac{1}{2} \frac{\partial}{\partial t} \big\|v(t)\big\|_{H^1(\R^d)}^2 + \int_{\R^d}  \Big(\phi(u^{\Delta t}(t)) -\phi(u^{\Delta s}(t))\Big)w\,dx 
  - \int_{\R^d}  \Big(\phi(u^{\Delta t}(t)) -\phi(u^{\Delta s}(t))\Big)v\,dx \notag \\
  & + \eps \int_{\R^d} (u^{\Delta t}-u^{\Delta s})w\,dx -  \eps \int_{\R^d} (u^{\Delta t}-u^{\Delta s})v\,dx
  +  \int_{\R^d}\Big( f(u^{\Delta t})- f(u^{\Delta s})\Big)\cdot \grad v \,dx =0. \label{variational_formula_discrete_3}
\end{align}
Note that, $w= (u^{\Delta t}-u^{\Delta s}) - (\tilde{B}^{\Delta t}-\tilde{B}^{\Delta s})- (u^{\Delta t}-\tilde{u}^{\Delta t})
 +  (u^{\Delta s}-\tilde{u}^{\Delta s}) $. 
 
 Therefore,
 \begin{align}
  &\int_{\R^d}  \Big(\phi(u^{\Delta t}(t)) -\phi(u^{\Delta s}(t))\Big)w\,dx \notag  \\
  =&\int_{\R^d}  \Big(\phi(u^{\Delta t}(t)) -\phi(u^{\Delta s}(t))\Big) (u^{\Delta t}-u^{\Delta s})\,dx 
   - \int_{\R^d}  \Big(\phi(u^{\Delta t}(t))- \phi(u^{\Delta s}(t))\Big)(\tilde{B}^{\Delta t}-\tilde{B}^{\Delta s}) \,dx \notag \\
  & \qquad  -\int_{\R^d}  \Big(\phi(u^{\Delta t}(t))- \phi(u^{\Delta s}(t))\Big) \Big\{(u^{\Delta t}-\tilde{u}^{\Delta t})+ 
  (\tilde{u}^{\Delta s}-u^{\Delta s} )\Big\}\,dx \notag  \\
  \ge & \int_{\R^d}  \Big(\phi(u^{\Delta t}(t)) -\phi(u^{\Delta s}(t))\Big) (u^{\Delta t}-u^{\Delta s})\,dx 
   - \frac{1}{2c_\phi}\int_{\R^d}  \Big(\phi(u^{\Delta t}(t))- \phi(u^{\Delta s}(t))\Big)^2 \,dx \notag \\
   & \qquad -\frac{c_\phi}{2} \int_{\R^d} \Big\{ (\tilde{B}^{\Delta t}-\tilde{B}^{\Delta s}) +(u^{\Delta t}-\tilde{u}^{\Delta t})+
   (\tilde{u}^{\Delta s}-u^{\Delta s} )\Big\}^2\,dx \notag \\
   \ge & \frac{1}{2}\int_{\R^d}  \Big(\phi(u^{\Delta t}(t)) -\phi(u^{\Delta s}(t))\Big) (u^{\Delta t}-u^{\Delta s})\,dx \notag \\
   & \qquad -\frac{c_\phi}{2} \int_{\R^d} \Big\{ (\tilde{B}^{\Delta t}-\tilde{B}^{\Delta s}) +(u^{\Delta t}-\tilde{u}^{\Delta t})+
   (\tilde{u}^{\Delta s}-u^{\Delta s} )\Big\}^2\,dx \notag \\  
   \ge & -\frac{c_\phi}{2} \int_{\R^d} \Big\{ (\tilde{B}^{\Delta t}-\tilde{B}^{\Delta s}) +(u^{\Delta t}-\tilde{u}^{\Delta t})+ 
   (\tilde{u}^{\Delta s}-u^{\Delta s} )\Big\}^2\,dx\quad (\text{by}\,eq\ref{A1}).\label{variational_formula_discrete_estimate-1}
 \end{align}
 
 Similarly, we also have
 \begin{align}
  &\eps \int_{\R^d} (u^{\Delta t}-u^{\Delta s})w\,dx  \notag \\
  \ge &  \frac{\eps}{2} \int_{\R^d} (u^{\Delta t}-u^{\Delta s})^2\,dx 
  -\frac{\eps}{2} \int_{\R^d} \Big\{ (\tilde{B}^{\Delta t}-\tilde{B}^{\Delta s}) +(u^{\Delta t}-\tilde{u}^{\Delta t})+ 
  (\tilde{u}^{\Delta s}-u^{\Delta s} )\Big\}^2\,dx.\label{variational_formula_discrete_estimate-2}
 \end{align}
 Combining  \eqref{variational_formula_discrete_3}, \eqref{variational_formula_discrete_estimate-1} and \eqref{variational_formula_discrete_estimate-2} in 
 
 \begin{align}
   & \frac{1}{2} \frac{\partial}{\partial t} ||v(t)||_{H^1(\R^d)}^2  +  \frac{\eps}{2} \int_{\R^d} (u^{\Delta t}-u^{\Delta s})^2\,dx  \notag \\
   \le & \frac{(c_\phi+\eps)}{2}  \int_{\R^d} \Big\{ (\tilde{B}^{\Delta t}-\tilde{B}^{\Delta s}) +(u^{\Delta t}-\tilde{u}^{\Delta t})+
   (\tilde{u}^{\Delta s}-u^{\Delta s} )\Big\}^2\,dx\notag \\
   & +  \eps \int_{\R^d} (u^{\Delta t}-u^{\Delta s})v\,dx
  -  \int_{\R^d}\Big( F(u^{\Delta t})- F(u^{\Delta s})\Big)\cdot \grad v \,dx + \int_{\R^d}  \Big(\phi(u^{\Delta t}(t))
  -\phi(u^{\Delta s}(t))\Big)v\,dx \notag \\
  \le & \frac{(c_\phi+\eps)}{2}  \int_{\R^d} \Big\{ (\tilde{B}^{\Delta t}-\tilde{B}^{\Delta s}) +(u^{\Delta t}-\tilde{u}^{\Delta t})+
  (\tilde{u}^{\Delta s}-u^{\Delta s} )\Big\}^2\,dx  + \frac{\beta}{4} \int_{\R^d} |\grad v|^2 \,dx \notag \\
    +& \frac{ \eps}{2\alpha} \int_{\R^d} (u^{\Delta t}-u^{\Delta s})^2 dx +  \Big(\frac{\beta}{4} + \eps\,\frac{ \alpha}{2}\Big)\int_{\R^d} v^2\,dx
   + \frac{(c_f)^2 + (c_\phi)^2}{\beta} \int_{\R^d}(u^{\Delta t}- u^{\Delta s})^2\,dx \notag
 \end{align}
 for some $\alpha$ and $\beta >0$. Since $\alpha , \beta >0$ are arbitrary, there exist  positive constants $C_1, C_2$ and $C_3$ 
 such that
 \begin{align}
  &\E \Big[||v(t)||_{H^1(\R^d)}^2\Big] - C_1 \int_{0}^t  \E \Big[||v(r)||_{H^1(\R^d)}^2\Big]\,dr 
  + C_2 \int_0^t \E \Big[\big\|u^{\Delta t}-u^{\Delta s}\big\|_{L^2(\R^d)}^2\Big]\,dr \notag \\
  &  \le  C_3 \Bigg\{ \big\|u^{\Delta t}-\tilde{u}^{\Delta t}\big\|_{L^2(\Omega \times \Pi_T)}^2 + 
  \big\|u^{\Delta s}-\tilde{u}^{\Delta s}\big\|_{L^2(\Omega \times \Pi_T)}^2
  + \int_0^t \E \Big[\big\|\tilde{B}^{\Delta t}-\tilde{B}^{\Delta s}\big\|_{L^2(\R^d)}^2\Big]\,dr \Bigg\}.\label{variational_formula_discrete_4}
 \end{align}
 In view of the Proposition \ref{Prop: a-priori_bound_1}, we notice that
 \begin{align}
  \big\|u^{\Delta t}-\tilde{u}^{\Delta t}\big\|_{L^2(\Omega \times \Pi_T)}^2 + \big\|u^{\Delta s}-\tilde{u}^{\Delta s}\big\|_{L^2(\Omega 
  \times \Pi_T)}^2 \le C\big(\Delta t + \Delta s\big).\label{variational_formula_discrete_estimate-3}
 \end{align}
 So we need to estimate the term $ \int_0^t \E \Big[\big\|\tilde{B}^{\Delta t}-\tilde{B}^{\Delta s}\big\|_{L^2(\R^d)}^2\Big]\,dr $. Now
 \begin{align*}
  & \E \Big[\big\|\tilde{B}^{\Delta t}(r)-\tilde{B}^{\Delta s}(r)\big\|_{L^2(\R^d)}^2\Big] \notag \\
  \le & 3 \Big\| \tilde{B}^{\Delta t}(r) - \int_{0}^r \int_{E}  \eta(x,u^{\Delta t}(\sigma-\Delta t);z) \tilde{N}(dz,d\sigma)\Big\|_{L^2(\Omega \times \R^d)}^2 \notag \\
  & +  3\,\Big\| \tilde{B}^{\Delta s}(r) - \int_{0}^r \int_{E}  \eta(x,u^{\Delta s}(\sigma-\Delta s);z) \tilde{N}(dz,d\sigma)\Big\|_{L^2(\Omega \times \R^d)}^2 \notag \\
  & + 3\, \Big\|\int_{0}^r \int_{E}   \Big(\eta(x,u^{\Delta t}(\sigma-\Delta t);z)-\eta(x,u^{\Delta s}(\sigma-\Delta s);z)\Big)
  \tilde{N}(dz,d\sigma)\Big\|_{L^2(\Omega \times \R^d)}^2 \notag \\
  \le & C(\Delta t + \Delta s) + C \int_{0}^r \E \Big[ \big\|u^{\Delta t}(\sigma-\Delta t)-u^{\Delta s}
  (\sigma-\Delta s)\big\|_{L^2(\R^d)}^2\Big]\,d\sigma,
 \end{align*}
 where we have used Proposition \ref{Prop: a-priori_bound_2} and the assumption \ref{A5}. Thus, we get 
 \begin{align}
   & \int_0^t \E \Big[\big\|\tilde{B}^{\Delta t}-\tilde{B}^{\Delta s}\big\|_{L^2(\R^d)}^2\Big]\,dr \notag \\
   \le &  C\big(\Delta t + \Delta s\big) + C \int_0^t \int_0^r  \E \Big[ \big\|u^{\Delta t}(\sigma-\Delta t)-u^{\Delta s}
   (\sigma-\Delta s)\big\|_{L^2(\R^d)}^2\Big]\,d\sigma\,dr.\label{variational_formula_discrete_estimate-4}
 \end{align}
 We combine \eqref{variational_formula_discrete_estimate-3} and \eqref{variational_formula_discrete_estimate-4} in 
 \eqref{variational_formula_discrete_4} and have
 \begin{align*}
   &\E \Big[\big\|v(t)\big\|_{H^1(\R^d)}^2\Big] - C_1 \int_{0}^t  \E \Big[\big\|v(r)\big\|_{H^1(\R^d)}^2\Big]\,dr
   + C_2 \int_0^t \E \Big[\big\|u^{\Delta t}-u^{\Delta s}\big\|_{L^2(\R^d)}^2\Big]\,dr \notag \\
    & \le  C\big(\Delta t + \Delta s\big) +  C \int_0^t \int_0^r  \E \Big[ \big\|u^{\Delta t}(\sigma-\Delta t)-u^{\Delta s}
   (\sigma-\Delta s)\big\|_{L^2(\R^d)}^2\Big]\,d\sigma\,dr \notag \\
   & (\text{by Proposition \ref{Prop: a-priori_bound_1}}) \notag \\
   & \le C\big(\Delta t + \Delta s\big) +  C \int_0^t \int_0^r \E\Big[ \big\|u^{\Delta t}-u^{\Delta s}\big\|_{L^2(\R^d)}^2\Big]\,d\sigma\,dr.
 \end{align*}
 Hence, an application of the Gronwall's lemma gives
 \begin{align*}
   \E \Big[\big\|v(t)\big\|_{H^1(\R^d)}^2\Big]  +  \int_0^t \E\Big[\big\|u^{\Delta t}-u^{\Delta s}\big\|_{L^2(\R^d)}^2\Big]\,dr 
   \le & C\big(\Delta t + \Delta s\big)\, e^{Ct}.
 \end{align*}
 This implies that 
 \begin{align*}
 \big\|u^{\Delta t}-u^{\Delta s}\big\|_{L^2(\Omega \times \Pi_T)}^2 \le C\big(\Delta t + \Delta s\big)\, e^{CT}
 \end{align*}
 \textit{i.e.}, $\{u^{\Delta t}\}$ is a Cauchy sequence in $L^2(\Omega \times \Pi_T)$.
\end{proof}
\vspace{.2cm}

We are now in a position to identify the weak limits $\phi_{u}$ and $f_u$. We have shown that $u^{\Delta t} \rightharpoonup u$
and $u^{\Delta t}$ is a Cauchy sequence in $L^2(\Omega \times \Pi_T)$. Thanks to the Lipschitz continuity of $\phi$ and $f$, one can
easily conclude that $\phi_u = \phi(u)$ and $f_u = f(u)$.
\vspace{.2cm}
 
 In view of the variational formula \eqref{variational_formula_discrete_1}, one needs to show the boundedness of
 $\frac{\partial}{\partial t}(\tilde{u}^{\Delta t}-\tilde{B}^{\Delta t}) $ in 
$ L^2(\Omega\times (0,T);H^{-1}(\R^d))$ and then  identify  the weak limit. Regarding this, we have the following lemma.

\begin{lem}\label{lemma:weak-stochastic-term}
The sequence $\Big\{\frac{\partial}{\partial t}(\tilde{u}^{\Delta t}-\tilde{B}^{\Delta t})(t)\Big\}$ is  bounded in $L^2\big(\Omega\times (0,T);H^{-1}
 (\R^d)\big)$, and 
 \begin{align*}
 \frac{\partial}{\partial t}(\tilde{u}^{\Delta t}-\tilde{B}^{\Delta t}) \rightharpoonup
 \frac{\partial}{\partial t}\Big( u-\int_{0}^\cdot \int_{E} \eta(x,u;z)\tilde{N}(dz,ds)\Big)
 \quad \text{in}~~L^2\big(\Omega\times (0,T);H^{-1}(\R^d)\big)
 \end{align*}
 where $u$ is given by \eqref{convergence:weak-1}. 
\end{lem}

\begin{proof}
To prove the lemma, let $\Gamma = \Omega\times [0,T] \times E $, $\mathcal{G}= \mathcal{P}_T\times \mathcal{L}(E)$  and
 $\varsigma = P\otimes \ell_t \otimes m$, where $\mathcal{P}_T$ represents predictable $\sigma$-algebra on $\Omega \times [0,T]$
  and $\mathcal{L}(E)$ represents a Lebesgue $\sigma$- algebra on $E$. Then $L^2\big((\Gamma, \mathcal{G}, \varsigma); \R\big)$ consists of all
 square integrable predictable processes which are Borel measurable functions of $z$-variable.
 

 
The space   $L^2\big((\Gamma, \mathcal{G}, \varsigma); \R\big)$ represents the space of square integrable predictable integrands
for It\^{o}-L\'{e}vy integrals with respect to the compensated compound  Poisson random measure $\tilde{N}(\,dz, \, dt)$. 
Moreover, It\^{o}-L\'{e}vy integral defines a linear operator from  $L^2\big((\Gamma, \mathcal{G}, \varsigma); \R\big)$ to 
$L^2\big((\Omega, \mathcal{F}_T); \R\big)$ and it preserves the norm (cf. for example \cite{peszat}). 
\vspace{.1cm}
\\
Thank to Propositions \ref{Prop: a-priori_bound_1} and \ref{lemma:cauchyness},  $u^{\Delta t}(t-\Delta t)$ converges to $u$ in $L^2(\Omega \times \Pi_T)$. 
Therefore, in view of the Proposition \ref{Prop: a-priori_bound_2}, Lipschitz property of $\eta$, and the above discussion  we conclude that
 \begin{align*} 
  \tilde{B}^{\Delta t} \to \int_{0}^\cdot \int_{E} \eta(x,u;z)\tilde{N}(dz,ds) \quad \text{in}\,~~L^2(\Omega \times \Pi_T).
 \end{align*}
 

Again, note that
 \begin{align*}
 \frac{\partial}{\partial t}\big(\tilde{u}^{\Delta t}-\tilde{B}^{\Delta t}\big)(t)
 = \sum_{k=1}^N \frac{ (u_k -u_{k-1}) - (B_k -B_{k-1})}{\Delta t}{\bf 1}_{\big[(k-1)\Delta t, k\Delta t\big)}.
 \end{align*}
 From \eqref{variational_formula_discrete}, we see that for any $v\in H^1(\R^d)$,
 \begin{align*}
 & \int_{\R^d} \Big(\frac{u_{n+1}-u_n}{\Delta t} -\frac{1}{\Delta t} \int_{n\Delta t}^{(n+1)\Delta t} \int_{E} \eta(x,u_n;z)
  \tilde{N}(dz,ds)\Big)v\,dx \notag \\
  &= -\int_{\R^d} \grad \phi(u_{n+1})\cdot \grad v\,dx -\eps \int_{\R^d} \grad u_{n+1}\cdot \grad v\,dx -\int_{\R^d} F(u_{n+1})\cdot \grad v\,dx \notag \\
  &\le  \Big\{ \big\| \grad \phi(u_{n+1})\big\|_{L^2(\R^d)} + \eps \big\|\grad u_{n+1}\big\|_{L^2(\R^d)} +
  c_f \big\|u_{n+1}\big\|_{L^2(\R^d)}\Big\} ||v||_{H^1(\R^d)},
 \end{align*} 
 and hence 
 \begin{align*}
  & \sup_{v\in H^1(\R^d)\setminus \{0\}} \frac{\int_{\R^d} \Big( \frac{u_{n+1}-u_n}{\Delta t} -\frac{1}{\Delta t} \int_{n\Delta t}^{(n+1)\Delta t} \int_{E} \eta(x,u_n;z)
  \tilde{N}(dz,ds)\Big)v\,dx }{||v||_{H^1(\R^d)}} \notag \\
   & \hspace{2.5cm} \le \big\| \grad \phi(u_{n+1})\big\|_{L^2(\R^d)} + \eps \big\|\grad u_{n+1}\big\|_{L^2(\R^d)}
   + c_f \big\|u_{n+1}\big\|_{L^2(\R^d)}.
 \end{align*}
 This implies that $\frac{\partial}{\partial t}(\tilde{u}^{\Delta t}-\tilde{B}^{\Delta t})(t)$ is a bounded
 sequence in $L^2(\Omega\times (0,T);H^{-1}(\R^d))$.
 \vspace{.1cm}
 
 To prove the second part of the lemma, we recall that
  $ \tilde{B}^{\Delta t} \rightharpoonup \int_{0}^\cdot \int_{E} \eta(x,u;z)\tilde{N}(dz,ds)
  $  and $\tilde{u}^{\Delta t}  \rightharpoonup u $ in $L^2(\Omega \times \Pi_T)$. In view of the first part of this lemma, 
   one can conclude that, up to a subsequence 
 \begin{align*}
 \frac{\partial}{\partial t}(\tilde{u}^{\Delta t}-\tilde{B}^{\Delta t}) \rightharpoonup
 \frac{\partial}{\partial t}\Big( u-\int_{0}^\cdot \int_{E} \eta(x,u;z)\tilde{N}(dz,ds)\Big) \quad \text{in}~~L^2(\Omega\times (0,T);H^{-1}(\R^d)).
 \end{align*}
 This completes the proof.
 \end{proof}
 
 \subsubsection{ \bf Existence of weak solution}\label{Existence of weak solution}
 As we have emphasized, our aim is to prove the existence of weak solution for viscous problem. For this, it is required to pass the 
 limit as $\Delta t \goto 0$. To this end, let $\alpha \in L^2((0,T))$ and
 $\beta \in L^2(\Omega)$. Then, in view of variational formula \eqref{variational_formula_discrete_1}, we obtain
 \begin{align*}
   & \int_{\Omega \times (0,T)} \Big \langle \frac{\partial}{\partial t}(\tilde{u}^{\Delta t}-\tilde{B}^{\Delta t}),v \Big \rangle
    \alpha\, \beta \,dt\,dP
  + \eps  \int_{\Omega \times \Pi_T} \grad u^{\Delta t}\cdot \grad v\,\alpha \beta \, dx\,dt\,dP  \notag \\
  & \quad +   \int_{\Omega \times \Pi_T} \grad \phi( u^{\Delta t})\cdot \grad v\,\alpha \beta \, dx\,dt\,dP 
  +  \int_{\Omega \times \Pi_T} f( u^{\Delta t})\cdot \grad v\,\alpha \beta \, dx\,dt\,dP =0.
 \end{align*}
 We make use of \eqref{convergence:weak-1}, Lemmas \ref{lemma:cauchyness} and \ref{lemma:weak-stochastic-term} to pass 
  to the limit as $\Delta t \goto 0$ in the above variational formulation and arrive at
 \begin{align}
   & \int_{\Omega \times (0,T)} \Big \langle  \frac{\partial}{\partial t}\Big( u-\int_{0}^{t} \int_{E} 
   \eta(x,u;z)\tilde{N}(dz,ds)\Big),v \Big \rangle \alpha\, \beta \,dt\,dP + \eps  \int_{\Omega \times \Pi_T} \grad u\cdot \grad v\,\alpha \beta \, dx\,dt\,dP \notag \\
   & \hspace{3cm} + \int_{\Omega \times \Pi_T}\Big\{ \grad \phi( u)\cdot \grad v + f( u)\cdot \grad v\Big\} \alpha \beta \,dx\,dt\,dP =0.\label{weak-formulation-final}
 \end{align}
 Since $H^1(\R^d)$ is a separable Hilbert space, the above formulation \eqref{weak-formulation-final} yields for almost surely
 $ \omega \in \Omega$,
 \begin{align*}
\Big \langle  \frac{\partial}{\partial t}\Big( u-\int_{0}^t \int_{E} \eta(x,u;z)\tilde{N}(dz,ds)\Big),v \Big\rangle 
+ \int_{\R^d} \Big(\eps \, \grad u + \grad \phi(u) + f( u)\Big) \cdot \grad v \,dx=0,
\end{align*} 
 for almost every $t \in [0,T]$. 

 
 This proves that $u$ is a weak solution of \eqref{eq:levy_stochconservation_laws-viscous}.
 Note that, for every $\phi\in H^1(\R^d)$, it is easily seen that
 \begin{align*}
 \E  \big|  \frac{1}{\delta}\int_0^\delta\int_{\R^d} \big(u(s,x)-u_0(x)\big)\psi(x)\,dx\,ds \big| \le C\Delta t\quad \text{if}\quad \delta <\Delta t.
 \end{align*} Therefore, 
 \begin{align*}
 \limsup_{\delta\downarrow 0} \E \big|  \frac{1}{\delta}\int_0^\delta\int_{\R^d} \big(u(s,x)-u_0(x)\big)\psi(x)\,dx\,ds \big| \le C\Delta t\quad \text{for very }\quad \delta <\Delta t. \end{align*} Now, by letting $\Delta t\downarrow 0$, we have 
 \begin{align}
 \label{eq:initial-con-weak} \lim_{\delta\downarrow 0}  \frac{1}{\delta}\int_0^\delta\int_{\R^d} u(s,x)\psi(x)\,dx\,ds= \int_{\R^d}u_0(x) \psi(x)\,dx \quad P-\text{almost surely}.
  \end{align} 
 
 \subsection{\textit{A priori} bounds for viscous solutions}
 Note that for fixed $\eps>0$, there exists a weak solution, denoted as $u_\eps \in H^1(\R^d)$, which satisfies 
 the following variational formulation: $P$- almost surely in $\Omega$, and almost every $t\in (0,T)$,
\begin{align*}
& \Big\langle \frac{\partial}{\partial t} [u_\eps-\int_0^t \int_{E} \eta(\cdot,u_\eps(s,\cdot);z) \tilde{N}(dz,ds)],v \Big\rangle
+ \int_{\R^d} \grad \phi(u_\eps(t,x))\cdot \grad v\,dx \\
& \hspace{3cm} + \int_{\R^d}\Big\{ f(u_\eps(t,x)) + \eps \grad u_\eps (t,x)\Big\} \cdot\grad v\,dx=0,
\end{align*}
for any $v\in H^1(\R^d)$.
 Let $\beta(u)=u^2$. Applying It\^{o}-L\'{e}vy  formula (cf. Theorem \ref{thm:weak-its} in the Appendix) to $\beta(u)$, one gets that for any $t>0$
\begin{align*}
  & \int_{\R^d} u_\eps^2(t)\,dx + 2 \int_0^t \int_{\R^d} \big[\eps + \phi^\prime(u_\eps(s))\big] |\grad u_\eps|^2\,ds + 
  2\int_{\R^d} f(u_\eps)\cdot \grad u_\eps\,dx \notag \\
 &=  \int_{\R^d} u_\eps^2(0)\,dx +  2\,\int_0^t \int_{E}\int_{\R^d} \int_0^1 \eta(x,u_\eps(s,x);z) \big( u_\eps+\theta\,
 \eta(x,u_\eps;z)\big) \,d\theta\,dx\,\tilde{N}(dz,ds) \notag \\
  & \qquad \quad + \int_0^t \int_{E}\int_{\R^d}   \eta^2(x,u_\eps(s,x);z) \,dx\,m(dz)\,ds.
\end{align*}
Note that $\int_{\R^d} f(u_\eps)\cdot \grad u_\eps\,dx =0$ as $u_\eps \in H^1(\R^d)$. Taking expectation, we obtain
\begin{align*}
  & \E\Big[ \big\|u_\eps(t)\big\|_2^2\Big] + \eps \int_0^t \E\Big[\big\|\grad u_\eps\big\|_2^2\Big]\,ds  
  + \int_0^t  \E\Big[\big\|\grad G(u_\eps(s))\big\|_2^2\Big]\,ds  \notag \\
  & \quad  \le  \E\Big[ \big\|u_\eps(0)\big\|_2^2 + C t \big\|g\big\|_{L^2(\R^d)}^2 + C \int_0^t \E\Big[\big\|u_\eps(s)\big\|_2^2\Big]\,ds.
\end{align*}
An application of Gronwall's inequality yields
\begin{align*}
 \sup_{0\le t\le T}  \E\Big[\big\|u_\eps(t)\big\|_2^2\Big]  + \eps \int_0^T \E\Big[\big\|\grad u_\eps(s)\big\|_2^2\Big]\,ds
 + \int_0^T \E\Big[\big\|\grad G(u_\eps(s))\big\|_2^2\Big]\,ds \le C.
\end{align*}
\vspace{.2cm}

The achieved results can be summarized into the following theorem.
\begin{thm} \label{thm:existec-viscous-apriori}
 For any $\eps>0$, there exists a  weak solution $u_\eps$ to the problem \eqref{eq:levy_stochconservation_laws-viscous}. Moreover, it
 satisfies the following estimate:
  \begin{align}
   \sup_{0\le t\le T}  \E\Big[\big\|u_\eps(t)\big\|_2^2\Big]  + \eps \int_0^T \E\Big[\big\|\grad u_\eps(s)\big\|_2^2\Big]\,ds
 + \int_0^T \E\Big[\big\|\grad G(u_\eps(s))\big\|_2^2\Big]\,ds \le C,\label{bounds:a-priori-viscous-solution}
  \end{align} 
  where $G$ is an associated Kirchoff's function of $\phi$, defined by $G(x)= \int_0^x \sqrt{\phi^\prime(r)}\,dr$.
 \end{thm}
\begin{rem} Let us remark that since any solution to \eqref{eq:levy_stochconservation_laws-viscous} is an entropy solution,  the solution $u_\epsilon$ in unique.
\end{rem}

\section{Existence of entropy solution}\label{sec:existence-entropy}
In this section, we will prove the existence of entropy solution.  In view of the  \textit{a priori} estimates as in \eqref{bounds:a-priori-viscous-solution}, we can apply Lemmas $4.2$ and 
$4.3$ of \cite{BisMajKarl_2014}(see also \cite{BaVaWit_2012}) and show the existence of Young measure-valued limit process solution $u(t,x,\alpha), \alpha \in (0,1)$ 
associated to the sequence $\{u_\eps(t,x)\}_{\eps>0}$.

The basic strategy in this case is to apply Young measure technique and adapt Kruzkov's
doubling method in the presence of noise for viscous solutions with two different parameters and then send the viscous parameters goes to zero.
One needs a version of classical $L^1$ contraction principle( for conservation laws) to get the uniqueness of Young measure valued limit and show that Young measure valued limit 
 process is independent of the additional (dummy) variable and hence it will imply the point-wise convergence of viscous solutions.
 
\subsection{Uniqueness of Young measure valued limit process}
 To do this, we follow the same line of argument as in \cite{BaVaWit_2014} for the degenerate parabolic part and \cite{BisMajKarl_2014} for the L\'evy noise. For the convenience of the reader, we have chosen to provide detailed proofs of a few crucial technical lemmas and the rest are referred to the appropriate resources.
 In \cite{BaVaWit_2012,BisMajKarl_2014},
 the authors used the fact that $\Delta u_\eps \in L^2(\Omega \times \Pi_T)$. Note that, in this case, $u_\eps \in H^1(\R^d)$. Therefore, we need to regularize $u_\eps$ by convolution.
 Let $ \{\tau_\kappa\} $ be a sequence of mollifiers in $\R^d$. Since $u_\eps$ is a viscous solution to the problem 
 \eqref{eq:levy_stochconservation_laws-viscous}, as shown in the proof of Theorem \ref{thm:weak-its},
 $u_\eps \con \tau_\kappa$ is a solution to the problem
\begin{align}
   (u_\eps \con \tau_\kappa) -\int_0^t\Delta (\phi(u_\eps)\con \tau_\kappa )\,ds =\int_0^t& \mbox{div}_x (f(u_\eps)\con \tau_\kappa) \,ds
  +\int_0^t\int_{E} (\eta(x, u_\eps;z)\con \tau_\kappa )\tilde{N}(dz,ds) \notag \\
 & +\int_0^t \eps \Delta( u_\eps \con \tau_\kappa(t,x))\,ds \quad \text{a.e.} ~t>0, ~ x\in \R^d.
 \label{eq:levy_stochconservation_laws-viscous-regularize}
\end{align}
Note that, for fixed $\eps>0$, $ \Delta (u_\eps \con \tau_\kappa) \in L^2(\Omega \times \Pi_T)$.
\vspace{.2cm}

Let $\rho$ and $\varrho$ be the standard nonnegative mollifiers on $\R$ and $\R^d$ respectively such that
  $\supp(\rho) \subset [-1,0]$ and 
  \guy{
  $\supp(\varrho) = \overline B_1(0)$. 
  }
  We define $\rho_{\delta_0}(r) = \frac{1}{\delta_0}\rho(\frac{r}{\delta_0})$ and 
  $\varrho_{\delta}(x) = \frac{1}{\delta^d}\varrho(\frac{x}{\delta})$, where $\delta$ and $\delta_0$ are two positive constants.
  Given a nonnegative test function $\psi\in C_c^{1,2}([0,\infty)\times \rd)$ and two positive constants $\delta$ and
  $ \delta_0 $, define
 \begin{align}
   \varphi_{\delta,\delta_0}(t,x, s,y) = \rho_{\delta_0}(t-s) \varrho_{\delta}(x-y) \psi(s,y). \label{eq:doubled-variable}
 \end{align}
 Clearly $ \rho_{\delta_0}(t-s) \neq 0$ only if $s-\delta_0 \le t\le s$ and hence $\varphi_{\delta,\delta_0}(t,x; s,y)= 0$ outside
 $s-\delta_0 \le t\le s$.
 \vspace{.1cm}
 
Let $u_{\theta}(t,x)$ be a weak solution to the viscous problem \eqref{eq:levy_stochconservation_laws-viscous} with parameter
$\theta >0$ and initial condition $u_{\theta}(0,x)=v_0(x)$. Moreover, let $J$ be the standard symmetric nonnegative mollifier on $\R$
with support in $[-1,1]$ and $ J_l(r)= \frac{1}{l} J(\frac{r}{l})$ for $l > 0$. We now simply write down It\^{o}-L\'{e}vy formula
for $u_{\theta}(t,x)$ against the convex entropy triple $(\beta(\cdot-k), F^\beta(\cdot, k),\phi^\beta(\cdot,k))$ and 
then multiply by $J_l(u_\eps\con \tau_\kappa(s,y)-k)$ for $k\in \R$
and then integrate with respect to $ s, y, k$. Taking expectation to the resulting equations, we have
\begin{align}
& \theta  \E\Big[\int_{\Pi_T}\int_{\Pi_T} \int_{\R} \beta^{\prime\prime}(u_{\theta}(t,x)-k) |\grad u_{\theta}(t,x)|^2 
\varphi_{\delta,\delta_0}(t,x,s,y)  J_l(u_\eps\con \tau_\kappa(s,y)-k)\,dk\,dx\,dt\,dy\,ds\Big] \notag \\
 & + \E\Big[\int_{\Pi_T}\int_{\Pi_T} \int_{\R} \beta^{\prime \prime }(u_{\theta}(t,x)-k) |\Grad G(u_{\theta}(t,x))|^2
 \varphi_{\delta,\delta_0}(t,x,s,y)J_l(u_\eps\con \tau_\kappa(s,y)-k)\,dk\,dx\,dt\,dy\,ds\Big] \notag \\
\le  & \E\Big[\int_{\Pi_T}\int_{\R^d}\int_{\R} \beta(v_0(x)-k)\varphi_{\delta,\delta_0}(0,x,s,y) J_l(u_\eps\con \tau_\kappa(s,y)-k)\,dk \,dx\,dy\,ds\Big] \notag \\
  +& \E\Big[ \int_{\Pi_T} \int_{\Pi_T}\int_{\R} \beta(u_{\theta}(t,x)-k)\partial_t \varphi_{\delta,\delta_0}(t,x,s,y)
J_l(u_\eps\con \tau_\kappa(s,y)-k)\,dk \,dx\,dt\,dy\,ds\Big]\notag \\ 
 + &\E\Big[\int_{\Pi_T}\int_{\R} \int_{0}^T\int_{E} \int_{\R^d} \int_0^1 \eta(x,u_{\theta}(t,x);z) \beta^\prime (u_{\theta}(t,x)
 +\tau\,\eta(x,u_{\theta}(t,x);z)-k) \notag \\
& \hspace{4cm} \times \varphi_{\delta,\delta_0}(t,x,s,y) \,d\tau\,dx\,\tilde{N}(dz,dt)J_l(u_\eps\con \tau_\kappa(s,y)-k)\,dk \,dy\,ds\Big]\notag\\
 + & \E\Big[\int_{\Pi_T} \int_{0}^T\int_{E}\int_{\R^d} \int_{\R} \int_0^1 (1-\tau) \eta^2(x,u_{\theta}(t,x);z)
 \beta^{\prime\prime}(u_{\theta}(t,x) +\tau\,\eta(x,u_{\theta}(t,x);z)-k) \notag \\
&\hspace{4.5cm} \times \varphi_{\delta,\delta_0}(t,x,s,y) J_l(u_\eps\con \tau_\kappa(s,y)-k)\,d\tau\,dk\,dx\,m(dz)\,dt\,dy\,ds\Big]\notag \\
 -&\E\Big[\int_{\Pi_T}\int_{\Pi_T} \int_{\R}  F^\beta(u_{\theta}(t,x),k)\grad_x \varrho_\delta(x-y)\,\psi(s,y)\,\rho_{\delta_0}(t-s)
J_l(u_\eps\con \tau_\kappa(s,y)-k)\,dk\,dx\,dt\,dy\,ds\Big] \notag \\
+& \E\Big[\int_{\Pi_T}\int_{\Pi_T} \int_{\R}  \phi^\beta(u_{\theta}(t,x),k)\Delta_x \varrho_\delta(x-y)\,\psi(s,y)\,\rho_{\delta_0}(t-s)
J_l(u_\eps\con \tau_\kappa(s,y)-k)\,dk\,dx\,dt\,dy\,ds\Big] \notag \\
-& \theta \E\Big[\int_{\Pi_T}\int_{\Pi_T} \int_{\R} \beta^{\prime}(u_{\theta}(t,x)-k) \grad_x u_{\theta}(t,x)\cdot \grad_x \varphi_{\delta,\delta_0}(t,x,s,y)
J_l(u_\eps\con \tau_\kappa(s,y)-k)\,dk\,dx\,dt\,dy\,ds\Big] \notag \\
 & \text{ i.e.,} \qquad  I_{0,1}+ I_{0,2} \le  I_1 + I_2 + I_3 +I_4 + I_5 + I_6 + I_7.\label{stochas_entropy_1}
\end{align}
We now apply It\^{o}-L\'{e}vy formula to \eqref{eq:levy_stochconservation_laws-viscous-regularize} and obtain
\begin{align}
&\eps  \E\Big[\int_{\Pi_T}\int_{\Pi_T} \int_{\R} \beta^{\prime \prime}(u_\eps\con \tau_\kappa -k) |\grad(u_\eps\con \tau_\kappa)|^2
\varphi_{\delta,\delta_0} J_l(u_{\theta}(t,x)-k)\,dk\,dx\,dt\,dy\,ds\Big] \notag \\
+& \E\Big[\int_{\Pi_T}\int_{\Pi_T} \int_{\R} \beta^{\prime \prime}(u_\eps\con \tau_\kappa -k)
\grad( \phi(u_\eps)\con \tau_\kappa)\cdot \grad(u_\eps\con \tau_\kappa)
\varphi_{\delta,\delta_0} J_l(u_{\theta}(t,x)-k)\,dk\,dx\,dt\,dy\,ds\Big] \notag \\
  \le &  \E\Big[\int_{\Pi_T}\int_{\R^d}\int_{\R} \beta(u_\eps\con \tau_\kappa(0,y)-k)\varphi_{\delta,\delta_0}
  (t,x,0,y) J_l(u_{\theta}(t,x)-k)\,dk\,dx\,dy\,dt\Big] \notag \\
   +& \E \Big[ \int_{\Pi_T} \int_{\Pi_T} \int_{\R} \beta(u_\eps\con \tau_\kappa(s,y)-k)\partial_s \phi_{\delta,\delta_0}
 J_l(u_{\theta}(t,x)-k)\,dk \,dy\,ds\,dx\,dt\Big]\notag \\ 
  +& \E \Big[\int_{\Pi_T} \int_{0}^T\int_{E} \int_{\R} \int_{\R^d} \int_0^1  (\eta(y,u_\eps;z)\con \tau_\kappa)
 \beta^\prime(u_\eps\con \tau_\kappa +\theta (\eta(y,u_\eps;z)\con \tau_\kappa)-k)\notag \\
 & \hspace{4cm} \times \varphi_{\delta,\delta_0}J_l(u_{\theta}(t,x)-k)\,d\theta\,dy\,dk \,\tilde{N}(dz,ds)\,dx\,dt\Big]\notag\\
  +&  \E \Big[\int_{\Pi_T} \int_{0}^T\int_{E} \int_{\R} \int_{\R^d} \int_0^1 (1-\theta)  (\eta(y,u_\eps;z)\con \tau_\kappa)^2
 \beta^{\prime\prime}(u_\eps\con \tau_\kappa +\theta (\eta(y,u_\eps;z)\con \tau_\kappa)-k)\notag \\
 & \hspace{5cm} \times \varphi_{\delta,\delta_0}J_l(u_{\theta}(t,x)-k)\,d\theta\,dy\,dk\,m(dz) \,ds\,dx\,dt\Big]\notag \\
  - & \E \Big[\int_{\Pi_T}\int_{\Pi_T} \int_{\R}\beta^\prime(u_\eps\con \tau_\kappa-k) \grad(\phi(u_\eps)\con \tau_\kappa) 
  \grad_y \varphi_{\delta,\delta_0}  J_l(u_{\theta}(t,x)-k)\,dk\,dx\,dt\,dy\,ds\Big] \notag \\
  -& \E \Big[\int_{\Pi_T}\int_{\Pi_T} \int_{\R} \beta^\prime(u_\eps\con \tau_\kappa-k)(F(u_\eps)\con \tau_\kappa) \grad_y
 \varphi_{\delta,\delta_0} J_l(u_{\theta}(t,x)-k)\,dk\,dx\,dt\,dy\,ds\Big] \notag \\
   -& \E \Big[\int_{\Pi_T}\int_{\Pi_T} \int_{\R} \beta^{\prime\prime}(u_\eps\con \tau_\kappa-k)(F(u_\eps)\con \tau_\kappa) \grad_y
 (u_\eps \con \tau_\kappa) \varphi_{\delta,\delta_0} J_l(u_{\theta}(t,x)-k)\,dk\,dx\,dt\,dy\,ds\Big] \notag \\
  -&\eps  \E \Big[\int_{\Pi_T} \int_{\Pi_T} \int_{\R} \beta^\prime(u_\eps\con \tau_\kappa(s,y)-k)
 \grad_y (u_\eps\con \tau_\kappa)\cdot\grad_y  \varphi_{\delta,\delta_0}
 J_l(u_{\theta}(t,x)-k)\,dk \,dy\,ds\,dx\,dt\Big] \notag \\
 & \text{i.e.,} \qquad J_{0,1} + J_{0,2} \le  J_1 + J_2 + J_3 + J_4 + J_5 + J_6 + J_7 + J_8. \label{stochas_entropy_2}
\end{align} 
We now add \eqref{stochas_entropy_1} and \eqref{stochas_entropy_2} and look for the passage to the limit in various parameter 
involved. Note that $I_{0,1}$ and $J_{0,1}$ are both positive terms and they are left hand side of the inequalities \eqref{stochas_entropy_1}
and \eqref{stochas_entropy_2} respectively. Hence we can omit these terms. Let us consider the expressions $I_{0,2}$ and $J_{0,2}$. Recall that
\begin{align*}
I_{0,2} &=\E \Big[\int_{\Pi_T}\int_{\Pi_T} \int_{\R} \beta^{\prime \prime }(u_{\theta}(t,x)-k) |\Grad G(u_{\theta}(t,x))|^2 \varphi_{\delta,\delta_0}
  J_l(u_\eps\con \tau_\kappa(s,y)-k)\,dk\,dx\,dt\,dy\,ds\Big] 
  \end{align*}
  and 
  \begin{align*}
  J_{0,2} & = \E \Big[\int_{\Pi_T}\int_{\Pi_T} \int_{\R} \beta^{\prime \prime}(u_\eps\con \tau_\kappa -k)
\grad( \phi(u_\eps)\con \tau_\kappa)\cdot \grad(u_\eps\con \tau_\kappa)
\varphi_{\delta,\delta_0} J_l(u_{\theta}(t,x)-k)\,dk\,dx\,dt\,dy\,ds\Big].
\end{align*}
 By using the properties of Lebesgue points, convolutions, and approximations by mollifications 
one can able to pass to the limit in $I_{0,2}$ and $J_{0,2}$, and conclude the following lemma.
\begin{lem}\label{stochastic_lemma_1}
It holds that 
 \begin{align}
  \lim_{l\goto 0}\lim_{\kappa \goto 0}\lim_{\delta_0 \goto 0}I_{0,2} =  \E \Big[\int_{\Pi_T}\int_{\R^d} 
 \beta^{\prime \prime }(u_{\theta}(t,x)-u_\eps(t,y)) |\Grad G(u_{\theta}(t,x))|^2 \psi(t,y)\varrho_{\delta} (x-y)\,dx\,dt\,dy\,\Big] \notag 
\end{align} and 
\begin{align}
  & \lim_{l\goto 0}\lim_{\kappa \goto 0}\lim_{\delta_0 \goto 0}J_{0,2} 
   = \E\Big[\int_{\Pi_T}\int_{\R^d}  \beta^{\prime \prime}\big(u_\eps(s,y)-u_{\theta}(s,x)\big)\grad \phi(u_\eps(s,y))\cdot \grad u_\eps(s,y)\notag \\
   & \hspace{5cm} \times \psi(s,y)\varrho_{\delta}(x-y)\,dx\,dy\,ds\Big].\notag
\end{align}
\end{lem}
\begin{lem}\label{stochastic_lemma_additional_nondegeneracy}
 It follows that
 \begin{align}
& \limsup_{\theta \goto 0}\,\limsup_{\eps \goto 0}\, \lim_{l\goto 0}\lim_{\kappa \goto 0}\lim_{\delta_0 \goto 0} \big(I_{0,2} + J_{0,2} \big) \notag \\
& \ge  2\,\E\Big[\int_{\Pi_T}\int_{\R^d} \int_0^1 \int_0^1 \int_{\tilde{u}(t,x,\gamma)}^{u(t,y,\alpha)}
 \Big(\int_{s=\sigma}^{\tilde{u}(t,x,\gamma)} \beta^{\prime \prime}(\sigma-s) \sqrt{\phi^\prime(s)}\,ds\Big)\sqrt{\phi^\prime(\sigma)}\,d\sigma \notag \\
& \hspace{6cm} \times \textrm{div}_y \grad_x \big[\psi(t,y)\varrho_{\delta}(x-y)\big] \,d\gamma\,d\alpha\,dx\,dy\,dt\Big]. \label{inq:additional_nondegeneracy}
 \end{align}
\end{lem}

\begin{proof}
In view of Lemma \ref{stochastic_lemma_1}, we see that
\begin{align}
& \lim_{l\goto 0}\lim_{\kappa \goto 0}\lim_{\delta_0 \goto 0} \big(I_{0,2} + J_{0,2}\big) \notag \\
= & \E\Big[\int_{\Pi_T}\int_{\R^d}  \beta^{\prime \prime}(u_\eps(t,y)-u_{\theta}(t,x))
 \Big(|\grad_y G(u_\eps(t,y))|^2 + |\grad_x G(u_{\theta}(t,x))|^2\Big) \psi(t,y)\varrho_{\delta}(x-y)\,dx\,dy\,dt\Big].\notag
\end{align} 
 Let $u(t,y,\alpha)$ and  $\tilde{u}(t,x,\gamma)$ be Young measure-valued narrow limit associated to the sequences $\{u_\eps(t,y)\}_{\eps>0}$ and 
 $\{u_\theta(t,x)\}_{\theta>0}$ respectively. With these at hand, one can use a similar argument as in \cite[Lemma 3.4]{BaVaWit_2014} and arrive at
 the conclusion that \eqref{inq:additional_nondegeneracy} holds.
\end{proof}

Let us consider the terms  $(I_1 + J_1)$ coming from initial conditions. Note that $I_1=0$ as supp $\rho_{\delta_0} \subset [-\delta_0, 0)$. 
Under a slight modification of the same line arguments as in \cite{BisMajKarl_2014}, we arrive at the following lemma
\begin{lem}\label{stochastic_lemma_initialcond}
It holds that 
 \begin{align}
 \lim_{\theta \goto 0} \lim_{\eps \goto 0} \lim_{l\goto 0} \lim_{\kappa \goto 0}\lim_{\delta_0 \goto 0} \big(I_1 + J_1\big)
 = \E\Big[\int_{\R^d}\int_{\R^d} \beta_{\vartheta}\big(v_0(x)-u_0(y)\big)\,\psi(0,y)\varrho_{\delta} (x-y)\,dx \,dy\Big]\notag 
\end{align}
and
\begin{align}
 \lim_{(\vartheta, \delta)\goto (0,0)} & \E\Big[\int_{\R^d}\int_{\R^d} \beta_{\vartheta} \big(v_0(x)-u_0(y)\big)\,\psi(0,y)\varrho_{\delta} (x-y)\,dx \,dy \Big]\notag \\
&= \E\Big[\int_{\R^d} \big|v_0(x)-u_0(x)\big|\,\psi(0,x)\,dx\Big]. \notag
\end{align}
\end{lem}

We now turn our attention to $(I_2 + J_2)$. Note that $\partial_t \rho_{\delta_0}(t-s)= - \partial_s \rho_{\delta_0}(t-s)$ and 
$\beta, J_l$ are even functions. A simple calculation gives 
 \begin{align*}
  & I_2+ J_2 = \E\Big[\int_{\Pi_T}\int_{\Pi_T} \int_{\R} \beta\big(u_\eps \con \tau_\kappa(s,y)-k\big) \partial_s\psi(s,y)\, \rho_{\delta_0}(t-s)
  \varrho_\delta(x-y) \\
  & \hspace{5cm} \times J_l(u_{\theta}(t,x)-k)\,dk\,dy\,ds\,dx\,dt\Big]. 
 \end{align*}
 One can pass to the limit in $(I_2 + J_2)$ and have the following conclusion.
\begin{lem}\label{stochastic_lemma_partial_time}
  It holds that
  \begin{align*}
& \lim_{\theta \goto 0} \lim_{\eps \goto 0} \lim_{l\goto 0} \lim_{\kappa \goto 0}\lim_{\delta_0 \goto 0} \big(I_2 + J_2\big) \\
= & \E \Big[\int_{\Pi_T}\int_{\R^d}\int_0^1 \int_0^1 \beta\big(u(s,y,\alpha)-\tilde{u}(s,x,\gamma)\big) \partial_s\psi(s,y)\,
  \varrho_\delta(x-y)\,d\gamma \,d\alpha \,dy\,dx\,ds\Big],\\
  \end{align*} 
  and 
  \begin{align*}
  \lim_{(\vartheta,\delta)\goto (0,0)} & \E\Big[\int_{\Pi_T}\int_{\R^d}\int_0^1\int_{0}^1 \beta\big(u(s,y,\alpha)-\tilde{u}(s,x,\gamma)\big) \partial_s\psi(s,y)\,
  \varrho_\delta(x-y)\,d\gamma \,d\alpha \,dy\,dx\,ds\Big] \\
  & = \E\Big[\int_0^T \int_{\R^d}\int_0^1 \int_{0}^1 \big|u(s,y,\alpha)-\tilde{u}(s,y,\gamma)\big| \partial_s\psi(s,y)\,d\gamma\,d\alpha\,dy\,ds\Big].
  \end{align*}
  \end{lem}
  Let us consider $I_6$ and $J_5$. Regarding this, we have the following lemma
  \begin{lem} \label{stochastic_lemma_nondegeneracy}
  It holds that
   \begin{align}
&   \lim_{\theta \downarrow 0} \lim_{\eps\downarrow 0} \lim_{l\downarrow 0} \lim_{\kappa \downarrow 0} \lim_{\delta_0 \downarrow 0} I_6\notag\\&=
     \E \Big[\int_{\Pi_T}\int_{\R^d}\int_0^1 \int_0^1 \phi^\beta \big(\tilde{u}(s,x,\gamma),u(s,y,\alpha)\big)\Delta_x \varrho_\delta(x-y)
     \psi(s,y)\,d\gamma\, d\alpha \,dx\,dy\,ds\Big]\notag
   \end{align}
   and
    \begin{align}
 &  \lim_{\theta \downarrow 0} \lim_{\eps\downarrow 0} \lim_{l\downarrow 0} \lim_{\kappa \downarrow 0} \lim_{\delta_0 \downarrow 0} J_5\notag\\
   =& \E \Big[\int_{\Pi_T}\int_{\R^d} \int_{[0,1]^2}\phi^\beta\big(u(s,y,\alpha),\tilde{u}(s,x,\gamma)\big)
  \Delta_y \big[\psi(s,y)\varrho_\delta(x-y)\big] \,d\gamma\, d\alpha\, dx \,dy \,ds\Big].\notag
   \end{align}
  \end{lem}
  \begin{proof}
  Let us consider the passage to the limits in $I_6$. To do this, we define,
  \begin{align}
   \mathcal{B}_1=& \E \Big[\int_{\Pi_T}\int_{\Pi_T} \int_{\R}  \phi^\beta(u_{\theta}(t,x),k)\Delta_x \varrho_\delta(x-y)\,\psi(s,y)\,\rho_{\delta_0}(t-s)
J_l(u_\eps\con \tau_\kappa(s,y)-k)\,dk\,dx\,dt\,dy\,ds\Big] \notag \\
-& \E \Big[\int_{\Pi_T}\int_{\R^d} \int_{\R}  \phi^\beta(u_{\theta}(s,x),k)\Delta_x \varrho_\delta(x-y)\psi(s,y)
J_l(u_\eps\con \tau_\kappa(s,y)-k)\,dk\,dx\,dy\,ds\Big] \notag 
\end{align} 
Note that, for all $a, b,c\in \R$,
\begin{align}
|\phi^{\beta}(a,b)-\phi^{\beta}(c,b)|\le C|c-a|. \label{estimatephi}
\end{align}
By using \eqref{estimatephi}, we have
\begin{align}
\mathcal{B}_1
=& \E \Big[\int_{\Pi_T}\int_{\Pi_T} \int_{\R}  \phi^\beta(u_{\theta}(t,x),u_\eps\con \tau_\kappa(s,y)-k)\Delta_x \varrho_\delta(x-y)\,\psi(s,y)\,\rho_{\delta_0}(t-s)
J_l(k)\,dk\,dx\,dt\,dy\,ds\Big] \notag \\
-& \E \Big[\int_{\Pi_T}\int_{\R^d} \int_{\R}  \phi^\beta(u_{\theta}(s,x),u_\eps\con \tau_\kappa(s,y)-k)\Delta_x \varrho_\delta(x-y)\psi(s,y)
J_l(k)\,dk\,dx\,dy\,ds\Big] \notag \\
=& \E \Big[\int_{\Pi_T}\int_{\Pi_T} \int_{\R}  \Big(\phi^\beta \big(u_{\theta}(t,x),u_\eps\con \tau_\kappa(s,y)-k\big)-
 \phi^\beta \big(u_{\theta}(s,x),u_\eps\con \tau_\kappa(s,y)-k\big)\Big)
\Delta_x \varrho_\delta(x-y)\notag \\
& \hspace{3cm} \times \psi(s,y)\,\rho_{\delta_0}(t-s)J_l(k)\,dk\,dx\,dt\,dy\,ds\Big] \notag \\
-& \E \Big[\int_{\Pi_T}\int_{\R^d} \int_{\R}  \phi^\beta \big(u_{\theta}(s,x),u_\eps\con \tau_\kappa(s,y)-k\big)
\psi(s,y) \Big(1-\int_0^T \rho_{\delta_0}(t-s)\,dt\Big)  \notag \\
\mbox{Then}&\hspace{7cm} \times \Delta_x \varrho_\delta(x-y)J_l(k)\,dk\,dx\,dy\,ds\Big]. \notag \\
|\mathcal{B}_1|
\le & C \E \Big[\int_{s=\delta_0}^T \int_{\Pi_T}\int_{\R^d} \int_{\R}|u_{\theta}(t,x)-u_{\theta}(s,x)|
|\Delta_x \varrho_\delta(x-y)| \psi(s,y)\,\rho_{\delta_0}(t-s)J_l(k)\,dk\,dx\,dt\,dy\,ds\Big] \notag \\
& \hspace{6cm}+ \mathcal{O}(\delta_0)\notag \\
\le & C \E \Big[\int_{s=\delta_0}^T \int_{0}^T\int_{K_\delta} |u_{\theta}(t,x)-u_{\theta}(s,x)|\rho_{\delta_0}(t-s)
dx\,dt\,ds\Big]  + \mathcal{O}(\delta_0)\notag \\
\le & C \E \Big[\int_{r=0}^1 \int_{0}^T\int_{K_\delta}  |u_{\theta}(t+\delta_0\,r,x)-u_{\theta}(t,x)|\rho(-r)
dx\,dt\,dr\Big] + \mathcal{O}(\delta_0)\notag 
 \end{align}
where $K_\delta \subset \R^d$ is a compact set depending on $\psi$ and $\delta$.
\\
  Note that, $\underset{\delta_0\downarrow 0 }\lim\, \int_0^T\int_{K_\delta} |u_{\theta}(t+\delta_0 r,x)-u_{\theta}(t,x)|\,dx\,dt \rightarrow 0$ 
  almost surely for all $ r\in [0,1] $. Therefore, by dominated convergence theorem, we have
\begin{align*}
 \lim_{\delta_0 \goto 0} I_6= \E \Big[\int_{\Pi_T}\int_{\R^d} \int_{\R}  \phi^\beta(u_{\theta}(s,x),k)\Delta_x \varrho_\delta(x-y)\psi(s,y)
J_l(u_\eps\con \tau_\kappa(s,y)-k)\,dk\,dx\,dy\,ds\Big].
\end{align*}
Moreover, one can use the property of convolution to conclude
\begin{align*}
\lim_{\kappa \goto 0} \lim_{\delta_0 \goto 0} I_6= \E \Big[\int_{\Pi_T}\int_{\R^d} \int_{\R}  \phi^\beta(v(s,x),k)\Delta_x \varrho_\delta(x-y)\psi(s,y)
J_l(u_\eps(s,y)-k)\,dk\,dx\,dy\,ds\Big].
\end{align*}
\vspace{.2cm}

Passage to the limit as $l\goto 0$: 
\begin{align}
 \text{Let,}~~\mathcal{B}_2:= & \E \Big[\int_{\Pi_T}\int_{\R^d} \int_{\R}  \phi^\beta(u_{\theta}(s,x),k)\Delta_x \varrho_\delta(x-y)\psi(s,y)
J_l(u_\eps(s,y)-k)\,dk\,dx\,dy\,ds\Big]\notag \\
-&  \E \Big[\int_{\Pi_T}\int_{\R^d}  \phi^\beta(u_{\theta}(s,x),u_\eps(s,y))\Delta_x \varrho_\delta(x-y)\psi(s,y)\,dx\,dy\,ds\Big] \notag \\
&= \E \Big[\int_{\Pi_T}\int_{\R^d} \int_{\R} \Big( \phi^\beta(u_{\theta}(s,x),k)- \phi^\beta(u_{\theta}(s,x),u_\eps(s,y))\Big)\Delta_x \varrho_\delta(x-y)\psi(s,y)\notag \\
 & \hspace{3cm} \times J_l(u_\eps(s,y)-k)\,dk\,dx\,dy\,ds\Big].\notag 
\end{align}
Note that for all $a, b,c\in \R$,
\begin{align}
|\phi^{\beta}(a,b)-\phi^{\beta}(a,c)|\le C(1+|a-b|)|b-c|. \label{estimatephi1}
\end{align}
Therefore, by \eqref{estimatephi1} we have
\begin{align}
 |\mathcal{B}_2|\le & C\,\E \Big[\int_{\Pi_T}\int_{\R^d} \int_{\R} \big(1+ |u_{\theta}(s,x)-k|\big)|u_\eps(s,y)-k|\, |\Delta_x \varrho_\delta(x-y)|\psi(s,y)\notag \\
 & \hspace{3cm} \times J_l(u_\eps(s,y)-k)\,dk\,dx\,dy\,ds\Big]\notag \\
 & \le C\,l\,\Big\{ 1+ \Big(\sup_{\theta>0}\sup_{t>0} \E\big[||u_{\theta}(t)||_2^2\big]\Big)^\frac{1}{2}
 + \Big(\sup_{\eps>0}\sup_{t>0} \E\big[||u_\eps(t)||_2^2\big]\Big)^\frac{1}{2}\Big\} 
 \longrightarrow 0 \quad \text{as}\quad l \goto 0.\notag
\end{align}
\vspace{.2cm}

One can justify the passage to the limit as $\eps \goto 0$ and $\theta \goto 0$ in the sense of Young measures
as in \cite{BaVaWit_2012,BisMajKarl_2014}  and  conclude 
\begin{align}
 & \lim_{\theta \rightarrow 0}   \lim_{\eps\rightarrow 0}  \E \Big[\int_{\Pi_T}\int_{\R^d}  \phi^\beta \big(u_{\theta}(s,x), u_\eps(s,y)\big)\Delta_x \varrho_\delta(x-y)\psi(s,y)\,dx\,dy\,ds\Big]\notag \\
 =&  \E \Big[\int_{\Pi_T}\int_{\R^d}\int_0^1 \int_0^1 \phi^\beta \big(\tilde{u}(s,x,\gamma),u(s,y,\alpha)\big)
 \Delta_x \varrho_\delta(x-y)\psi(s,y)\,d\gamma\,d\alpha\,dx\,dy\,ds\Big].\notag
\end{align}
This proves the first part of the lemma.
\vspace{.1cm}

To prove the second part, let us recall that 
\begin{align*}
 J_5:= - \E \Big[\int_{\Pi_T}\int_{\Pi_T} \int_{\R}\beta^\prime(u_\eps\con \tau_\kappa-k) \grad(\phi(u_\eps)\con \tau_\kappa) 
  \grad_y \varphi_{\delta,\delta_0}  J_l(u_{\theta}(t,x)-k)\,dk\,dx\,dt\,dy\,ds\Big]
\end{align*}
A classical properties of Lebesgue points and convolution yields 
\begin{align*}
 \lim_{\kappa \goto 0}\lim_{\delta_0 \goto 0} J_5 = &
 - \E \Big[\int_{\Pi_T}\int_{\R^d} \int_{\R}\beta^\prime(u_\eps(s,y)-k) \grad \phi(u_\eps(s,y))\cdot 
  \grad_y \big[\psi(s,y)\varrho_{\delta}(x-y)\big] \notag \\
  & \hspace{3cm} \times J_l \big(u_{\theta}(s,x)-k\big)\,dk\,dx\,dy\,ds\Big].
\end{align*}
Recall that $\phi^{\beta}(a,b)=\int_b^a \beta^{\prime}(s-b)\phi^\prime(s)\,ds$. Making use of Green's type formula along with 
Young measure theory and keeping in mind that $u(s,y,\alpha)$ and $\tilde{u}(s,x,\gamma)$ are Young measure-valued limit
processes associated to the sequences $\{u_{\eps}(s,y)\}_{\eps>0}$ and $\{u_{\theta}(s,x)\}_{\theta>0}$ respectively, we arrive at
the following conclusion 
\begin{align*}
 &\lim_{\theta \downarrow 0} \lim_{\eps\downarrow 0} \lim_{l\downarrow 0} \lim_{\kappa \goto 0} \lim_{\delta_0 \goto 0} J_5 \\
   &= \E\Big[\int_{\Pi_T}\int_{\R^d} \int_{0}^1 \int_0^1\phi^\beta\big(u(s,y,\alpha),\tilde{u}(s,x,\gamma)\big)
  \Delta_y \big[\psi(s,y)\varrho_\delta(x-y)\big] \,d\gamma\,d\alpha \,dx\,dy\,ds\Big].\notag
\end{align*}
This completes the proof of the lemma. 
\end{proof}
 Next, we want to pass to the limits in $ \big(J_6 + J_7\big)$ and $I_5$ respectively. 
 A slight modification of the similar argument as in \cite{BaVaWit_2014,BaVaWit_2012,BisMajKarl_2014} yields the following lemma:
 \begin{lem} \label{stochastic_lemma_2}
   It holds that
 \begin{align*}
 &\lim_{l\goto 0} \lim_{\kappa \goto 0}\lim_{\delta_0 \goto 0}  \big(J_6+ J_7 \big) \\
  =& -\E \Big[\int_{\Pi_T}\int_{\R^d} F^\beta(u_\eps(s,y),u_\theta(s,x))\cdot \grad_y [ \psi(s,y) \varrho_\delta(x-y)]\,dx\,dy\,ds\Big] \\
  \underset{\eps \goto 0}\longrightarrow & -\E \Big[\int_{\Pi_T}\int_{\R^d}\int_0^1  F^\beta(u(s,y,\alpha),u_\theta(s,x)) \cdot\grad_y \big[ \psi(s,y)
 \varrho_\delta(x-y)\big]\,d\alpha\,dx\,dy\,ds\Big] \\
 \underset{\theta \goto 0}\longrightarrow & -\E \Big[\int_{\Pi_T}\int_{\R^d}\int_0^1 \int_0^1  F^\beta \big(u(s,y,\alpha),\tilde{u}(s,x,\gamma)\big) 
 \cdot\grad_y \big[\psi(s,y)\varrho_\delta(x-y)\big]\,d\gamma\, d\alpha\,dx\,dy\,ds\Big],
 \end{align*}
 and 
 \begin{align}
 &\lim_{\theta \goto 0} \lim_{\eps \goto 0} \lim_{l\goto 0} \lim_{\kappa \goto 0}\lim_{\delta_0 \goto 0}  I_5\\
 = & -\E \Big[\int_{\Pi_T}\int_{\R^d} \int_0^1 \int_0^1 F^\beta \big(\tilde{u}(s,x,\gamma),u(s,y,\alpha)\big) \cdot
 \grad_x \big[\varrho_\delta(x-y)\,\psi(s,y)\big]\,d\gamma \, d\alpha\,dx\,dy\,ds\Big]. \notag 
 \end{align}
\end{lem}

In view of the above, we now  want to pass the limit as $(\vartheta,\delta) \longrightarrow(0,0)$. We use similar line argument as in the proof 
 of the second part of the Lemmas $5.7$ and $5.8$ in \cite{BisMajKarl_2014} and arrive at the following lemma.
\begin{lem} \label{stochastic_lemma_3} 
  Assume  that $ \vartheta \goto 0, \delta \goto 0$ and $\frac{\vartheta}{\delta} \goto 0$, then
  \begin{align*}
   & \lim_{ \vartheta \goto 0} \lim_{\delta \goto 0 }\lim_{\frac{\vartheta}{\delta} \goto 0} \lim_{\theta \downarrow 0}  \lim_{\eps \downarrow 0}
   \lim_{l \goto 0} \lim_{\kappa \goto 0} \lim_{\delta_0 \goto 0} \Big[ \big(J_6 +J_7\big)+ I_5 \Big]\\
    &= - \E\Big[\int_{\Pi_T}\int_0^1 \int_0^1 F\big(u(s,y,\alpha), \tilde{u}(s,y,\gamma)\big) \cdot \grad_y \psi(s,y)\,d\gamma\,d\alpha\,dy\,ds\Big].
  \end{align*}
\end{lem}

Let us consider the term $I_7 + J_8$. Regarding this, we have the following
\begin{lem}\label{stochastic_lemma_4} For fixed $\delta > 0$ and $\beta$, it holds that
 \begin{align}
  \limsup_{( \theta, \eps,l,\kappa,\delta_0,)\rightarrow 0}\,  \big|I_7 + J_8 \big| = 0.
 \end{align} 
\end{lem}
 \begin{proof}
 Note that 
  \begin{align}
|J_8| &\le \eps ||\beta^\prime||_{\infty} \E\Big[\int_{\Pi_T}\int_{\R^d}|\grad_y (u_\eps\con \tau_\kappa(s,y))| |\grad_y[\psi(s,y)
\varrho_\delta(x-y)|\,dy\,dx\,ds\Big]\notag \\
&\le\eps  ||\beta^\prime||_{\infty} \E\Big[\int_{|y|\le K}\int_{0}^T\int_{\R^d} \,|\grad_y( u_\eps \con \tau_\kappa(t,y))|\,|\grad_y[\psi(t,y)\varrho_\delta(x-y)]|\,dx\,dt\,dy\Big]\notag \\
& (\text{By Cauchy-Schwartz inequality}),\notag \\
& \le C(\beta) \eps^{\frac 12}\Big(\E\Big[\int_{\Pi_T} \eps |\grad_y( u_\eps \con \tau_\kappa(t,y))|^2 \,dy\,dt\Big]\Big)^\frac{1}{2}
\Big( \E\Big[\int_{K}\int_{\Pi_T} |
\grad_y[\psi(t,y)\varrho_\delta(x-y)]|^2\,dx\,dt\,dy\Big]\Big)^\frac{1}{2}\notag \\
& \le  C(\beta,\psi,\delta)\,\eps^{\frac{1}{2}} \Big(\sup_{\eps>0} \E\Big[|\eps \int_0^T \int_{\R^d} |\grad_y u_\eps(t,y)|^2 \,dy\,dt|\Big]\Big)^\frac{1}{2}
\le  C(\beta,\psi,\delta)\,\eps^{\frac{1}{2}}.\notag
\end{align}
Similarly, we have $|I_7| \le C(\beta,\psi,\delta)\,\theta^{\frac{1}{2}}$. This completes the proof.
\end{proof}

 Next we consider the stochastic terms $ I_3 + J_3$. To this end, we cite \cite{BisMajKarl_2014}  and assert that for two constants $T_1, T_2\ge 0$ with $T_1<T_2$,

 \begin{align}
  \E \Big[X_{T_1}\int_{T_1}^{T_2} \int_{E}J(t,z)\tilde{N}(dz,dt)\Big] = 0\label{eq:conditional_indep}
\end{align}
where $J$ is a predictable integrand and $X$ is an adapted process.

For each $\beta \in C^\infty(\R)$ with $\beta^\prime, 
 \beta^{\prime\prime}\in C_b(\R)$ and nonnegative $\varphi\in C_c^\infty(\Pi_{\infty}\times \Pi_\infty)$, we define 
 \begin{align*}
 M[\beta, \varphi](s; y, v):=& \int_0^T\int_{E}\int_{\R^d} \Big\{\beta\big(u_{\theta}(r,x) +\eta(x,u_{\theta}(r,x);z)-v\big)
 -\beta \big(u_{\theta}(r,x)-v\big)\Big\}\\
 & \hspace{4cm} \times \varphi(r,x,s,y)\,dx\, \tilde{N}(dz,dr)
 \end{align*}
 where  $0\le s\le T$, $(y,v)\in \R^d \times \R$. Furthermore, we extend the process $u_\eps \con \tau_\kappa(\cdot,y)$ for negative time simply 
by $u_\eps \con \tau_\kappa(s,y) = u_{\eps}(0,\cdot)\con \tau_k (y)$ if $s< 0$.
With this convention, it follows from \eqref{eq:conditional_indep} that      
 \begin{align*}
 \E \Big[\int_{\R_v}\int_{\Pi_{T}} M[\beta,\varphi_{\delta,\delta_0}](s;y,v)\,J_l(u_\eps(s-\delta_0,y)-v)\,dy\,ds\,dv\Big]=0
 \end{align*}
and hence we have $J_3=0$ and 
 \begin{align}
  I_3 &= \E \Big[\int_{\R}\int_{\Pi_{T}} M[\beta,\varphi_{\delta,\delta_0}](s;y,v)\Big(J_l(u_\eps\con \tau_\kappa(s,y)-v)
  -J_l(u_\eps \con \tau_\kappa(s-\delta_0,y)-v)\Big)\,dy\,ds\,dv\Big]. \label{stochastic_estimate_6}
 \end{align} 
\vspace{.3cm}

Our aim is to pass to the limit  into the stochastic terms $I_3 + J_3$. For that, we need some estimate regarding
$M[\beta,\varphi_{\delta,\delta_0}]$, a proof of which could be found in \cite{BisMajKarl_2014}.

 \begin{lem} \label{lem:L-infinity estimate}
  Let $\gamma\in C^\infty(\R)$ be function such that $\gamma^\prime \in C_c^\infty(\R) $  and $p$ be a
  positive integer of the form $p= 2^k$ for some $k\in \mathbb{N}$.
  If $p\geq d+3$ then there exists a constant $C=C(\gamma^\prime, \psi,\delta )$
  \begin{align}
  \sup_{0\le s\le T}\Bigg( \E\Big[|| M[\gamma,\varphi_{\delta,\delta_0}](s;\cdot,\cdot)||_{L^\infty(\R^d\times \R)}^2\Big]\Bigg) \le
   \frac{C(\gamma', \psi,\delta )}{\delta_0^\frac{2(p-1)}{p}}.\label{eq:l-infinity bound}
  \end{align} and the following identities hold:
    \begin{align*}
     &\partial_v  M[\gamma, \varphi](s; y, v) =  M[-\gamma^\prime, \varphi](s; y, v)\\
     & \partial_{y_k}  M[\gamma, \varphi](s; y, v) =  M[\gamma, \partial_{y_k}\varphi](s; y, v).
    \end{align*}
 \end{lem}
 
  \begin{lem}\label{stochastic_lemma_itointegralterm} It holds that 
 \begin{align} 
  &\lim_{l\goto 0} \lim_{\kappa \goto 0}\lim_{\delta_0 \goto 0} I_3  \notag \\
  & = \E \Big[\int_{\Pi_{T}}\int_{\R^d}\int_{E} \Big\{ \beta \big(u_{\theta}(r,x)+ \eta(x,u_{\theta}(r,x);z)
  -u_\eps(r,y)-\eta(y,u_\eps(r,y);z)\big) \notag \\
  & \hspace{2.4cm}-\beta \big(u_{\theta}(r,x)-u_\eps(r,y)-\eta(y,u_\eps(r,y);z)\big)+ \beta \big(u_{\theta}(r,x)-u_\eps(r,y)\big) \notag \\
  &  \hspace{2.5cm}-\beta \big(u_{\theta}(r,x)+ \eta(x,u_{\theta}(r,x);z)-u_\eps(r,y)\big) 
  \Big\} \psi(r,y)\,\varrho_{\delta}(x-y)\,m(dz)\,dx\,dy\,dr\Big]. \notag
  \end{align}
 \end{lem}
 
 \begin{proof}
   Note that, for all $y\in \R^d$, $u_\eps \con \tau_\kappa(\cdot,y)$ satisfies
 \begin{align*}
  du_\eps \con \tau_\kappa(s,y)= & \text{div} (f(u_\eps) \con \tau_\kappa(s,y))\,ds + \Delta(\phi(u_\eps)\con \tau_\kappa(s,y))\,ds \notag \\
  & + \eps  \Delta u_\eps \con \tau_\kappa(s,y)\,ds+\int_{E} (\eta(\cdot,u_\eps;z) \con \tau_\kappa(s,y))\tilde{N}(dz,ds). 
 \end{align*}
 Now apply It\^{o}-L\'{e}vy formula on $J_l(u_\eps \con \tau_\kappa(s,y)-v)$ to get 
 \begin{align*}
  & J_l(u_\eps \con \tau_\kappa(s,y)-v)-J_l(u_\eps \con \tau_\kappa(s-\delta_0,y)-v)\\
  =& \int_{s-\delta_0}^s J_l^{\prime} (u_\eps \con \tau_\kappa(\sigma,y)-v)\Big(\text{div} (f(u_\eps) \con \tau_\kappa(\sigma,y)) + \eps  \Delta (u_\eps \con \tau_\kappa(\sigma,y))
  +\Delta(\phi(u_\eps)\con \tau_\kappa(\sigma,y)) \Big) \,d\sigma \\
  & + \int_{s-\delta_0}^s \int_{E} \Big( J_l(u_\eps \con \tau_\kappa(\sigma,y)+(\eta(\cdot,u_\eps;z) \con \tau_\kappa(\sigma,y))-v)- J_l(u_\eps \con \tau_\kappa(\sigma,y)-v)\Big)
  \tilde{N}(dz,d\sigma)\\
  & + \int_{s-\delta_0}^s \int_{E} \int_{\lambda=0}^1(1-\lambda)
  J_{l}^{\prime\prime}( u_\eps \con \tau_\kappa(\sigma,y)-v+\lambda\,(\eta(\cdot,u_\eps;z) \con \tau_\kappa(\sigma,y))) \\
  &  \hspace{3cm}\times  (\eta(\cdot,u_\eps;z) \con \tau_\kappa(\sigma,y))^2\,d\lambda \,m(dz)\,d\sigma\\
  =& - \frac{\partial}{\partial v}\int_{s-\delta_0}^s \Big(\text{div} (f(u_\eps) \con \tau_\kappa(\sigma,y)) + \eps  \Delta (u_\eps \con \tau_\kappa(\sigma,y))
  +\Delta(\phi(u_\eps)\con \tau_\kappa(\sigma,y)) \Big)\notag \\& \hspace{7cm} \times J_l (u_\eps \con \tau_\kappa(\sigma,y)-v)\,d\sigma \\
  & + \int_{s-\delta_0}^s \int_{E} \Big( J_l(u_\eps \con \tau_\kappa(\sigma,y)+(\eta(\cdot,u_\eps;z) \con \tau_\kappa(\sigma,y))-v)- J_l(u_\eps \con \tau_\kappa(\sigma,y)-v)\Big)
  \tilde{N}(dz,d\sigma)\\
  & + \int_{s-\delta_0}^s \int_{E} \int_{\lambda=0}^1(1-\lambda)
  J_{l}^{\prime\prime}( u_\eps \con \tau_\kappa(\sigma,y)-v+\lambda\,(\eta(\cdot,u_\eps;z) \con \tau_\kappa(\sigma,y))) \\
  &  \hspace{3cm}\times  (\eta(\cdot,u_\eps;z) \con \tau_\kappa(\sigma,y))^2\,d\lambda \,m(dz)\,d\sigma\\ 
 \end{align*}
  Therefore, from \eqref{stochastic_estimate_6} and Lemma \ref{lem:L-infinity estimate}, we have
  \begin{align}
 I_3 & = -\E \Big[\int_{\R}\int_{\Pi_{T}} M[\beta^{\prime},\varphi_{\delta,\delta_0}](s;y,v)\Big(
 \int_{s-\delta_0}^s J_l(u_\eps\con \tau_\kappa(\sigma,
 y)-v)\,\text{div} (f(u_\eps) \con \tau_\kappa(\sigma,y)) \,d\sigma\Big)\,ds\,dy\,dv\Big]\notag  \\
 & -\E \Big[\int_{\R}\int_{\Pi_{T}} M[\beta^{\prime},\varphi_{\delta,\delta_0}](s;y,v)\Big(\int_{s-\delta_0}^s
 J_l(u_\eps \con \tau_\kappa(\sigma,
 y)-v)\, \eps  \Delta( u_\eps\con \tau_\kappa(\sigma,y)) \,d\sigma\Big)\,ds\,dy\,dv\Big]\notag  \\
  & -\E \Big[\int_{\R}\int_{\Pi_{T}} M[\beta^{\prime},\varphi_{\delta,\delta_0}](s;y,v)\Big(\int_{s-\delta_0}^s
 J_l(u_\eps \con \tau_\kappa(\sigma,
 y)-v)\, \Delta( \phi(u_\eps)\con \tau_\kappa(\sigma,y)) \,d\sigma\Big)\,ds\,dy\,dv\Big]\notag  \\
 &  + \E \Big[\int_{\Pi_{T}}\int_{\R}\int_{s-\delta_0}^s\int_{\R^d}\int_{E} \Big( \beta(u_{\theta}(r,x)+ \eta(x,u_{\theta}(r,x);z)
  -v)-\beta(u_{\theta}(r,x)-v)\Big) \notag \\
  & \hspace{4cm}\times\Big( J_l(u_\eps\con \tau_\kappa(r,y)+\eta(\cdot,u_\eps;z)\con \tau_\kappa(r,y)-v)-J_l(u_\eps\con \tau_\kappa(r,y)-v)\Big) \notag \\
  & \hspace{5cm}\times\rho_{\delta_0}(r-s)\,\psi(s,y)\,\varrho_{\delta}(x-y) \,m(dz)\,dx\,dr\,dv\,dy\,ds\Big] \notag \\
  & +\E \Big[\int_{\R}\int_{\Pi_{T}} M[\beta,\phi_{\delta,\delta_0}](s;y,v)\Big\{ \int_{s-\delta_0}^s \int_{E}
  \int_{\lambda=0}^1 J_{l}^{\prime\prime}( u_\eps \con \tau_\kappa(\sigma,y)-v+\lambda\,(\eta(\cdot,u_\eps;z) \con \tau_\kappa(\sigma,y))) \notag  \\
  &  \hspace{6cm}\times (1-\lambda) (\eta(\cdot,u_\eps;z) \con \tau_\kappa(\sigma,y))^2\,d\lambda \,m(dz)\,d\sigma \Big\}\,dy\,ds\,dv\Big] \notag \\
  &\equiv A_1^{\kappa,l,\eps}(\delta,\delta_0) + A_2^{\kappa,l,\eps}(\delta,\delta_0)+  A_3^{\kappa,l,\eps}(\delta,\delta_0)
  + B^{\eps, l,\kappa} +  A_4^{\kappa,l,\eps}(\delta,\delta_0). \label{stochastic_estimate_7}
 \end{align}
 Note that, for fixed $\kappa$ and $\eps$, $\Delta \big(u_\eps\con \tau_\kappa \big) \in L^2(\Omega \times \Pi_T)$. One can use Young's 
 inequality for convolution and replace
 $u_\eps$ by $u_\eps \con \tau_\kappa$ to adapt the same line of argument as in \cite{BisMajKarl_2014} and conclude 
 \begin{align*}
  A_1^{\kappa,l,\eps}(\delta,\delta_0)\goto 0,\quad 
  A_2^{\kappa,l,\eps}(\delta,\delta_0)\goto 0,\quad 
   A_3^{\kappa,l,\eps}(\delta,\delta_0)\goto 0,\quad \text{and} \quad 
  A_4^{\kappa,l,\eps}(\delta,\delta_0)\goto 0 \quad \text{as} \,\, \delta_0 \goto 0.
 \end{align*}
Again it is routine to pass to the limit for $ B^{\eps, l,\kappa}$ and arrive at the conclusion that 
\begin{align}
 &\lim_{l\goto 0}\lim_{\kappa \goto 0}\lim_{\delta_{0}\goto 0} B^{\eps, l,\kappa}  \notag \\
 = & \E \Big[\int_{\Pi_{T}}\int_{\R^d}\int_{E} \Big\{ \beta \big(u_{\theta}(r,x)+ \eta(x,u_{\theta}(r,x);z)
  -u_\eps(r,y)-\eta(y,u_\eps(r,y);z)\big) \notag \\
  & \hspace{2.4cm}-\beta \big(u_{\theta}(r,x)-u_\eps(r,y)-\eta(y,u_\eps(r,y);z)\big)+ \beta \big(u_{\theta}(r,x)-u_\eps(r,y)\big) \notag \\
  &  \hspace{2.5cm}-\beta \big(u_{\theta}(r,x)+ \eta(x,u_{\theta}(r,x);z)-u_\eps(r,y)\big) 
  \Big\} \psi(r,y)\,\varrho_{\delta}(x-y)\,m(dz)\,dx\,dy\,dr\Big]. \notag
\end{align}
This completes the proof.
 \end{proof}
 
 Let us consider the additional terms $I_4 + J_4$. A similar line arguments as in \cite{BaVaWit_2014,BaVaWit_2012,BisMajKarl_2014} along with classical 
 properties of convolution yields the following:
 \begin{lem}\label{stochastic_lemma_additinalterm}
 \begin{align*}
  &\lim_{l\goto 0}\lim_{\kappa \goto 0}\lim_{\delta_{0}\goto 0} J_4 = \E \Big[ \int_{\Pi_T} \int_{\R^d}\int_{E}\int_{\lambda =0}^1 
 (1-\lambda)\beta^{\prime\prime} \big(u_\eps(s,y)-u_{\theta}(s,x)+\lambda \eta(y,u_\eps(s,y);z)\big)\notag \\
 &\hspace{5cm} \times\eta^2(y,u_\eps(s,y);z)\psi(s,y)\varrho_{\delta}(x-y)\,d\lambda \,m(dz)\,dx\,dy \,ds\Big],
\end{align*}
and
\begin{align*}
 &\lim_{l\goto 0}\lim_{\kappa \goto 0}\lim_{\delta_{0}\goto 0} I_4 = \E\Big[\int_{\Pi_T} \int_{\R^d}\int_{E}
 \int_{\lambda =0}^1 (1-\lambda)
 \beta^{\prime\prime} \big(u_{\theta}(s,x)-u_\eps(s,y) +\lambda \eta(x,u_{\theta}(s,x);z)\big)\notag \\
 &\hspace{5cm} \times \eta^2(x,u_{\theta}(s,x);z)\psi(s,x)\varrho_{\delta}(x-y) \,d\lambda \,m(dz)\,dx\,dy \,ds\Big].
\end{align*}
\end{lem}

 Now we add these terms (cf. Lemma \ref{stochastic_lemma_additinalterm}) with the terms coming from
 Lemma \ref{stochastic_lemma_itointegralterm} and have the following lemma.
 \begin{lem} \label{stochastic_lemma_itoint_additional}
   Assume that $ \vartheta \goto 0+ $, $ \delta\goto 0+ $ and $\vartheta ^{-1}\delta^{2}\goto 0+ $, then
\begin{align}
& \limsup_{\vartheta \goto 0+,\,\delta\goto 0+,\,\vartheta ^{-1}\delta^2\goto 0+ }\limsup_{\theta,\eps \rightarrow 0}\Big[
\lim_{l\goto 0} \lim_{\kappa \goto 0}\lim_{\delta_0 \goto 0}\Big((I_3 +J_3)+ (I_4 + J_4)\Big)\Big]=0.\notag
\end{align}
 \end{lem}
 \begin{proof}
  In view of Lemmas \ref{stochastic_lemma_itointegralterm} and \ref{stochastic_lemma_additinalterm}, we see that
  \begin{align}
& \lim_{l\goto 0}\lim_{\kappa \goto 0} \lim_{\delta_0 \goto 0} 
\Big((I_3 +J_3)+ (I_4 + J_4)\Big) \notag \\
&= \E \Big[\int_{\Pi_T}\int_{\R^d}\Big(\int_{E} \Big\{\beta \big(u_{\theta}(t,x) -u_\eps(t,y)+\eta(x,u_{\theta}(t,x);z)-\eta(y,u_\eps(t,y);z)\big)\notag \\
& \hspace{4cm} -\big(\eta(x,u_{\theta}(t,x);z)-\eta(y,u_\eps(t,y);z)\big)\beta^\prime \big(u_{\theta}(t,x)-u_\eps(t,y)\big) \notag \\
& \hspace{5cm}-\beta \big(u_{\theta}(t,x)-u_\eps(t,y)\big)\Big\}\,m(dz)\Big) \psi(t,y)\varrho_\delta(x-y)\,dx\,dy\,dt\Big] \notag \\
&=\E\Big[\int_{\Pi_T}\int_{\R^d}\Big( \int_{E}\int_{\tau=0}^1 b^2
(1-\tau)\beta^{\prime\prime}(a+\tau\,b)\,d\tau\,m(dz)\Big)\,\psi(t,y)\varrho_{\delta}(x-y)\,dx\,dy\,dt\Big],\label{eq:endgame}
\end{align}
where $ a=u_{\theta}(t,x)-u_\eps(t,y)$ and $b=\eta(x,u_{\theta}(t,x);z)-\eta(y,u_\eps(t,y);z)$. By using the similar argument as one used 
in the proof of  \cite[Lemma $5.11$]{BisMajKarl_2014}, we arrive at  
\begin{align*}
 \lim_{l\goto 0}\lim_{\kappa \goto 0} \lim_{\delta_0 \goto 0} 
\Big((I_3 +J_3)+ (I_4 + J_4)\Big) & \le  C_1\big(\vartheta + \vartheta^{-1}\delta^2\big)T,\notag
\end{align*}
where the constant $C_1$ depends only on $\psi$ and is in particular independent of $\eps$. 
We now let $\vartheta\goto 0$, $ \delta\goto 0$ and $\vartheta^{-1}\delta^{2}\goto0$, yielding
\begin{align*}
& \limsup_{\vartheta \goto 0,\delta\goto 0,\vartheta ^{-1}\delta^2\goto 0}
\limsup_{\theta,\eps \rightarrow 0}\Big[ \lim_{l\goto 0} \lim_{\kappa \goto 0}\lim_{\delta_0 \goto 0}
\Big((I_3 +J_3)+ (I_4 + J_4)\Big)\Big]\le 0.
\end{align*} 
This concludes the proof as $\underset{l\goto 0}\lim\, \underset{\kappa \goto 0} \lim\,\underset{\delta_0 \goto 0}
\lim\,\Big((I_3 +J_3)+ (I_4 + J_4)\Big) \ge 0$, thanks to \eqref{eq:endgame}. 
 \end{proof}

We now turn our attention back to the terms which are coming from the Lemmas \ref{stochastic_lemma_additional_nondegeneracy}
and \ref{stochastic_lemma_nondegeneracy}. To this end, define
\begin{align}
 \mathcal{H}:&= -2\E\Big[\int_{\Pi_T}\int_{\R^d} \int_0^1 \int_0^1 \int_{\tilde{u}(s,x,\gamma)}^{u(s,y,\alpha)}
 \Big(\int_{r=\sigma}^{\tilde{u}(s,x,\gamma)} \beta^{\prime \prime}(\sigma-r) \sqrt{\phi^\prime(r)}\,dr\Big)\sqrt{\phi^\prime(\sigma)}\,d\sigma  \notag\\
& \hspace{6cm} \times \mbox{div}_y \grad_x \big[\psi(s,y)\varrho_{\delta}(x-y)\big] \,d\gamma\,d\alpha\,dx\,dy\,ds\Big] \notag \\
& \quad + \E\Big[\int_{\Pi_T}\int_{\R^d}\int_0^1 \int_0^1  \phi^\beta \big(\tilde{u}(s,x,\gamma),u(s,y,\alpha)\big)\Delta_x \varrho_\delta(x-y)
 \psi(s,y)\,d\gamma\,d\alpha\,dx\,dy\,ds\Big] \notag \\
 & \qquad +  \E\Big[\int_{\Pi_T}\int_{\R^d} \int_0^1 \int_0^1 \phi^\beta \big(u(s,y,\alpha),\tilde{u}(s,x,\gamma)\big)
\Delta_y \big[\psi(s,y)\varrho_\delta(x-y)\big]\,d\gamma\,d\alpha \,dx\,dy\,ds\Big] \notag \\
&= \E\Big[\int_{\Pi_T}\int_{\R^d} \int_0^1 \int_0^1 \Big\{ 2I_\beta \big(\tilde{u}(s,x,\gamma),u(s,y,\alpha)\big)
+ \phi^\beta \big(\tilde{u}(s,x,\gamma),u(s,y,\alpha)\big)  \notag \\  
& \hspace{4cm}  + \phi^\beta \big(u(s,y,\alpha),\tilde{u}(s,x,\gamma)\big)\Big\}  
\psi(s,y)\Delta_y \varrho_\delta(x-y)\,d\gamma\,d\alpha\,dx\,dy\,ds\Big] \notag \\
& \quad + \E\Big[\int_{\Pi_T}\int_{\R^d} \int_0^1 \int_0^1 \Big( 2\,I_\beta \big(u(s,y,\alpha),\tilde{u}(s,x,\gamma)\big)
+ 2\,\phi^\beta \big(\tilde{u}(s,x,\gamma),u(s,y,\alpha)\big)\Big) \notag \\
& \hspace{6cm} \times \grad_y\psi(s,y)\cdot \grad_y \varrho_\delta(x-y)\,d\gamma\,d\alpha\,dx\,dy\,ds\Big] \notag \\
& \qquad \quad +  \E\Big[\int_{\Pi_T}\int_{\R^d} \int_0^1  \int_0^1 \phi^\beta \big(u(s,y,\alpha),\tilde{u}(s,x,\gamma)\big)
\Delta_y\psi(s,y) \varrho_\delta(x-y)\,d\gamma\,d\alpha\,dx\,dy\,ds\Big] \notag \\
& \equiv  \mathcal{H}_1 + \mathcal{H}_2 + \mathcal{H}_3,\label{eq:h}
\end{align}
where
\begin{align*}
 I_\beta(a,b)=\int_a^b \int_{\mu}^a \beta^{\prime\prime}(\mu-\sigma)\sqrt{\phi^\prime(\sigma)}\,d\sigma \sqrt{\phi^\prime(\mu)}\,d\mu 
 \quad \text{for any} \quad a,b \in \R.
\end{align*}
 Our aim is to pass the limit in $\mathcal{H}$ as $(\vartheta,\delta) \goto (0,0)$. For this, we need some \textit{a priori} estimates on $I_\beta(a,b)$. 
 Here we state the required lemma whose proof could be found in \cite{BaVaWit_2014}.
 \begin{lem}\label{lem:technical}
The following holds:
\begin{align}
 1.)~~ I_\beta(a,b)=  I_\beta(b,a) \quad \text{and} \quad  I_\beta(a,b)=- \frac{1}{2}  \int_a^b \int_{a}^{b} \beta^{\prime\prime}(\sigma-\mu)
  \sqrt{\phi^\prime(\mu)}\sqrt{\phi^\prime(\sigma)}\,d\mu \,d\sigma.  \notag \\
  2.)~~  2\, I_\beta(a,b) + \phi^\beta(a,b) + \phi^\beta(b,a)
   =\frac{1}{2}\int_a^b \int_{a}^{b} \beta^{\prime\prime}(\mu-\sigma) [\sqrt{\phi^\prime(\mu)}-\sqrt{\phi^\prime(\sigma)}]^2\,d\mu
   \,d\sigma.\notag
\end{align}
Moreover, if  $\sqrt{\phi^\prime}$ has a modulus of continuity $\omega_\phi$, then
\begin{align*}
 2\, I_\beta(a,b) + \phi^\beta(a,b) + \phi^\beta(b,a) \le C |b-a| |\omega_\phi(|\vartheta|)|^2
\end{align*}
and
\begin{align*}
 2\, I_\beta(a,b) + \phi^\beta(b,a) \le C |b-a|  |\omega_\phi(|\vartheta|)|^2 + C \min\{ 2\,\vartheta, |b-a| \}.
\end{align*}
\end{lem}
 Let us back to the expression $\mathcal{H}$. Regarding this, we have following:
 \begin{lem} \label{stochastic_lemma_6}
 \begin{align*}
   \lim_{(\vartheta,\delta) \goto (0,0)}  \mathcal{H}= \E\Big[\int_{\Pi_T} \int_0^1 \int_0^1 |\phi(u(s,y,\alpha)- \phi(\tilde{u}(s,y,\gamma))|
\Delta_y\psi(s,y)\,d\gamma\,d\alpha\,dy\,ds\Big].
\end{align*}
 \end{lem}
 
\begin{proof}
Let $\omega_{\phi}$ be a modulus of continuity of $\sqrt{\phi^\prime}$. Then, thanks to Lemma \ref{lem:technical}, we obtain
\begin{align}
 &|\mathcal{H}_1| \le C\, \E\Big[\int_{\Pi_T}\int_{\R^d}\int_0^1 \int_0^1 |\omega_\phi(|\vartheta|)|^2 
 \big|u(s,y,\alpha)-\tilde{u}(s,x,\gamma)\big| \psi(s,y) |\Delta_y \varrho_\delta(x-y)|\,d\gamma\,d\alpha\,dx\,dy\,ds\Big]\notag \\
   & \le  C(\psi) \frac{|\omega_\phi(|\vartheta|)|^2 }{\delta^2}, \notag 
  \end{align}
 and
 \begin{align}
  &|\mathcal{H}_2|\notag\\ \le & C\, \E\Big[\int_{\Pi_T}\int_{\R^d}\int_0^1 \int_0^1 |\omega_\phi(|\vartheta|)|^2
  \big|u(s,y,\alpha)-\tilde{u}(s,x,\gamma)\big|\, |\grad_y\psi(s,y)|\, |\grad_y \varrho_\delta(x-y)| \,d\gamma\,d\alpha\,dx\,dy\,ds\Big] \notag \\
   & \qquad + \E\Big[\int_{\Pi_T}\int_{\R^d} C \vartheta |\grad_y\psi(s,y)|\, |\grad_y
  \varrho_\delta(x-y)|\,dx\,dy\,ds\Big] \notag \\
  \le &  C(\psi) \frac{|\omega_\phi(|\vartheta|)|^2 }{\delta} + C \frac{\vartheta}{\delta}.\notag
\end{align}
Hence, we have
\begin{align}
  |\mathcal{H}_1| +  |\mathcal{H}_2|\le C(\psi) \frac{|\omega_\phi(|\vartheta|)|^2 }{\delta^2} 
  + C(\psi) \frac{|\omega_\phi(|\vartheta|)|^2 }{\delta} + C \frac{\vartheta}{\delta}.\notag
\end{align}
Put $\delta = \vartheta ^{\frac{2}{3}}$. Then, by our assumption \ref{A1}, we see that
\begin{align}
 \lim_{(\vartheta,\delta) \goto (0,0)} \big(\mathcal{H}_1 + \mathcal{H}_2\big)= 0.\notag
\end{align}
To prove the lemma, it is now required to show 
\begin{align*}
 \lim_{(\vartheta,\delta) \goto (0,0)} \mathcal{H}_3= \E\Big[\int_{\Pi_T} \int_0^1 \int_0^1 \big|\phi(u(s,y,\alpha))- \phi(u(s,y,\gamma))\big|
\Delta_y\psi(s,y) \,d\gamma\,d\alpha\,dy\,ds\Big]
\end{align*}
and this follows easily as $(a,b)\longmapsto |\phi(a)-\phi(b)| $ is Lipschitz continuous and
\begin{align*}
 \big| \phi^{\beta_\vartheta}(a,b)- |\phi(a)-\phi(b)| \big| \le C \vartheta \quad \text{for any }\quad a,b \in \R.
\end{align*}
\end{proof}


 All of the above results can now be combined into the following proposition.
\begin{prop} \label{kato_inequality}
 Let $\tilde{u}(t,x,\gamma)$ and $u(t,x,\alpha)$ be the predictable process with initial data $v(0,x)$ and $u(0,x)$ respectively 
 which have been extracted out of Young measure valued sub-sequential limit of the sequence  $\{u_{\theta}(t,x)\}_{\theta>0}$ and 
 $\{u_\eps(t,x)\}_{\eps>0}$ respectively. Then, for any non-negative $H^1( [0,\infty) \times \R^d)$ function $\psi(t,x)$ with compact
 support, the following inequality holds
  \begin{align}
  & 0\le  \E\Big[\int_{\R^d} \big|v_0(x)-u_0(x)\big|\psi(0,x)\,dx\Big] + \E\Big[\int_{\Pi_T}\int_0^1 \int_0^1  |\tilde{u}(t,x,\gamma)-u(t,x,\alpha)|
   \partial_t \psi(t,x)\,d\gamma\,d\alpha\,dx\,dt\Big] \notag \\
   & \quad - \E\Big[ \int_{\Pi_T} \int_0^1 \int_0^1 F\big(u(t,x,\alpha),\tilde{u}(t,x,\gamma)\big)\cdot\grad_x\psi(t,x)\,d\gamma\,d\alpha
   \,dx\,dt\Big] \notag \\
  & \qquad - \E\Big[ \int_{\Pi_T} \grad\Big( \int_0^1 \int_0^1 \big|\phi(u(t,x,\alpha))- \phi(\tilde{u}(t,x,\gamma))\big|\,d\gamma\,d\alpha\Big)
   \cdot\grad_x\psi(t,x)\,dx\,dt\Big]. \label{inq:kato-type}
  \end{align}
\end{prop}

\begin{proof} First we add \eqref{stochas_entropy_1} and \eqref{stochas_entropy_2} and then pass to the limit
$\underset{\eps \downarrow 0}\lim \,\underset{l\rightarrow 0} \lim\,\underset{\kappa \rightarrow 0}\lim
\,\underset{\delta_0\downarrow 0}\lim$. Invoking the 
Lemmas \ref{stochastic_lemma_additional_nondegeneracy}, \ref{stochastic_lemma_initialcond}, \ref{stochastic_lemma_partial_time},
\ref{stochastic_lemma_nondegeneracy}, 
\ref{stochastic_lemma_2}, \ref{stochastic_lemma_4} and \ref{stochastic_lemma_itoint_additional}, we put $\delta = \vartheta^{\frac 23}$
in the resulting expression and then let $\vartheta \goto 0$ with the second parts of 
Lemmas \ref{stochastic_lemma_initialcond}, \ref{stochastic_lemma_partial_time}. Keeping in mind the Lemmas
\ref{stochastic_lemma_3}, \ref{stochastic_lemma_itoint_additional} and \ref{stochastic_lemma_6}, we conclude that \eqref{inq:kato-type} holds 
for any nonnegative test function $\psi \in C_c^2\big([0,\infty)\times \R^d\big)$.
It now follows by routine approximation argument that \eqref{inq:kato-type} holds for any $\psi$ with compact support such 
that $\psi\in H^1( [0,\infty) \times \R^d)$. This completes the proof.
\end{proof}


\begin{rem}
Note that the same Proposition holds without assuming the existence of a modulus of continuity for $\phi^\prime$ if $\eta$ is not a function of $x$. Indeed, it is possible to pass to the limit first on the parameter $\vartheta$, then $\delta$, in Lemmas \ref{stochastic_lemma_initialcond}, \ref{stochastic_lemma_partial_time}, \ref{stochastic_lemma_3} and \ref{stochastic_lemma_itoint_additional}. 
Thus, if one assumes that $\eta$ is not a function of the space variable $x$, it is also possible to pass to the limit first on the parameter $\vartheta$, then $\delta$, in Lemma \ref{stochastic_lemma_itoint_additional} since one would have  
\begin{align*}
 \lim_{l\goto 0}\lim_{\kappa \goto 0} \lim_{\delta_0 \goto 0} 
\Big((I_3 +J_3)+ (I_4 + J_4)\Big) & \le  C_1 \vartheta T,\notag
\end{align*}
in its proof. Then, the result holds following \cite[first situation p.523]{BaVaWit_2014}.
\end{rem}

Our aim is to show the uniqueness of $u(t,x,\alpha)$ and $\tilde{u}(t,x,\gamma)$. To do this, here we follow the ideas
of \cite{boris_2010, BaVaWit_2014}, and define for each $ n\in \N $,
 \begin{align*}
\phi_n(x)=\begin{cases} 1,\quad \text{if}~~|x| \le n \\
                        \frac{n^a}{|x|^a} \quad \text{if}~~|x| > n 
           \end{cases}
\end{align*}
where $a=\frac{d}{2} + \tilde{\eps}$ in which $\tilde{\eps }>0$ could be chosen such a way that $\phi_n \in L^2(\R^d)$. Also, for each
$h>0$ and fixed $t\geq 0$, we define 
 \begin{align*}
\psi_h^t(s)=
\begin{cases}
1,\quad \text{if} ~ s\le t\\
1-\frac{s-t}{h}, \quad \text{if} ~t \le s\le t+h\\
0,\quad \text{if}~s> t+h.
\end{cases}
\end{align*}
 A straightforward calculation revels that,
 \begin {align*}
 \grad \phi_n(x) &= -a \frac{\phi_n(x)}{|x|} \frac{x}{|x|}{\bf 1}_{|x|>n} \in L^2(\R^d)^d  \\
\Delta  \phi_n(x)&= a(2+2 \tilde{\eps} -a)\frac{\phi_n(x)}{|x|^2} \in L^2(\{ |x|>n\}). 
\end{align*}
Clearly \eqref{inq:kato-type} holds with $\psi(s,x)=\phi_n(x)\psi_h^t(s)$. Thus, for a.e $t\ge 0$, we obtain
\begin{align}
   &\frac{1}{h}\int_{t}^{t+h} \E \Big[\int_{\R^d} \int_0^1 \int_0^1 \big|u(s,x,\alpha)-\tilde{u}(s,x,\gamma)|\phi_n(x)\,d\gamma\,d\alpha\,\,dx\Big]\,ds \notag \\
  & \le \E \Big[ \int_0^T \int_{\{|x|>n\}}\int_0^1 \int_0^1 |\phi(u(s,x,\alpha))- \phi(\tilde{u}(s,x,\gamma))|\Delta \phi_n(x)
  \psi_h^t(s)\,d\gamma\,d\alpha\,dx\,ds\Big] \notag \\ 
 & \quad - \E \Big[ \int_0^T \int_{\{|x|>n\}}\int_0^1 \int_0^1 F\big(u(s,x,\alpha),\tilde{u}(s,x,\gamma)\big)\cdot \grad \phi_n(x) 
  \psi_h^t(s)\,d\gamma\,d\alpha\,dx\,ds\Big] \notag \\
  &\qquad  - \E \Big[ \int_0^T \psi_h^t(s) \Big( \int_{\partial\{|x|>n\}}\int_0^1 \int_0^1 \big|\phi(u(s,x,\alpha))- \phi(\tilde{u}(s,x,\gamma))\big|\grad \phi_n(x)\cdot \tilde{n}
  \,d\gamma\,d\alpha\,dx\Big)\,ds\Big] \notag \\
  &  \qquad \qquad + \E\Big[\int_{\R^d} \big| v_0(x)-u_0(x) \big|\phi_n(x)\,dx\Big]. \label{inq:final-1}
 \end{align}
 Since $\grad \phi_n(x)\cdot \tilde{n}= \frac{a}{n} >0$ on the set $\partial\{|x|>n\}$, we have, from \eqref{inq:final-1}
 \begin{align}
   &\frac{1}{h}\int_{t}^{t+h} \E \Big[\int_{\R^d} \int_0^1 \int_0^1  \big|u(s,x,\alpha)-\tilde{u}(s,x,\gamma)\big|\phi_n(x)\,d\gamma\,d\alpha\,\,dx\Big]\,ds \notag \\
 &\le \E \Big[ \int_0^T \int_{\{|x|>n\}}\int_0^1 \int_0^1 a \big(2+2 \tilde{\eps} -a\big)\frac{\phi_n(x)}{|x|^2}
 \big|\phi(u(s,x,\alpha))- \phi(\tilde{u}(s,x,\gamma))\big| \psi_h^t(s)\,d\gamma\,d\alpha\,dx\,ds\Big] \notag \\ 
 &  \qquad + \E \Big[ \int_0^T \int_{\{|x|>n\}}\int_0^1 \int_0^1 a \frac{\phi_n(x)}{|x|} F\big(u(s,x,\alpha),\tilde{u}(s,x,\gamma)\big)\cdot  \frac{x}{|x|}
  \psi_h^t(s)\,d\gamma\,d\alpha\,dx\,ds\Big] \notag \\
  &  \qquad \qquad \quad+   \E\Big[\int_{\R^d} \big|v_0(x)-u_0(x)\big|\phi_n(x)\,dx\Big].\label{inq:final-2}
 \end{align}
 Note that $|F(a,b)|\le c_f|a-b|$ for any $a, b \in \R$. Since $\phi$ is Lipschitz continuous function and $n\ge 1$, inequality 
 \eqref{inq:final-2} gives
 \begin{align}
   &\frac{1}{h}\int_{t}^{t+h} \E \Big[\int_{\R^d} \int_0^1 \int_0^1  \big|u(s,x,\alpha)-\tilde{u}(s,x,\gamma)\big|\phi_n(x)\,d\gamma\,d\alpha\,\,dx\Big]\,ds \notag \\
 \le & C \E \Big[ \int_{\Pi_T} \int_{[0,1]^2} \big|u(s,x,\alpha)-\tilde{u}(s,x,\gamma)\big|\phi_n(x)\psi_h^t(s)\,d\gamma\,d\alpha\,dx\,ds\Big]
   + \E\Big[\int_{\R^d} \big|v_0(x)-u_0(x)|\phi_n(x)\,dx\Big]. \notag 
  \end{align}
  Taking limit as $h\goto 0$, and then using a weaker version of Gronwall's inequality, we obtain, for a.e. $t>0$,
   \begin{align}
   & \E \Big[\int_{\R^d} \int_0^1 \int_0^1 \big|u(t,x,\alpha)-\tilde{u}(t,x,\gamma)\big|\phi_n(x)\,d\gamma\,d\alpha\,dx\Big]
  \le  e^{CT} \,\E \Big[\int_{\R^d} \big|v_0(x)-u_0(x)\big|\phi_n(x)\,dx\Big].\notag
 \end{align}
 Thus, if we assume that $v_0(x)=u_0(x)$, then we arrive at the conclusion
\begin{align}
  \E \Big[\int_{\R^d} \int_0^1 \int_0^1 \big|u(t,x,\alpha)-\tilde{u}(t,x,\gamma)|\phi_n(x)\,d\gamma\,d\alpha\,dx\Big]=0,\label{eq:lasteqn}
 \end{align}
 which says that for almost all $\omega \in \Omega$, a.e. $(t,x)\in (0,T]\times \R^d$ and a.e. $(\alpha,\gamma)\in [0,1]^2$, $u(t,x,\alpha)=\tilde{u}(t,x,\gamma)$.  On the other hand,
 we conclude that the whole sequence of viscous approximation converges weakly in $L^2(\Omega \times \Pi_T)$. Since the limit process is 
 independent of the  additional (dummy) variable, the viscous approximation converges strongly in 
$L^p(\Omega \times (0,T); L^p(\Theta))$ for any $p<2$ and any bounded open set $\Theta \subset \R^d$.
 \subsection{Existence of entropy solution}
 In this subsection, using strong convergence of viscous solutions and \textit{a priori} bounds \eqref{bounds:a-priori-viscous-solution}
 we establish the existence of entropy solution to the underlying problem \eqref{eq:levy_stochconservation_laws}.
 \vspace{.2cm}
 
  Fix a nonnegative test function $ \psi\in C_c^\infty([0, \infty)\times \R^d)$, $B\in \mathcal{F}_T$ and convex entropy flux triple $(\beta,\zeta,\nu)$.
Now apply It\^{o}-L\'{e}vy formula \eqref{eq:levy_stochconservation_laws-viscous} and conclude
 \begin{align}
 & \E \Big[ {\bf 1}_{B} \int_{\Pi_T} \beta^{\prime\prime}(u_\eps(t,x)) |\Grad G(u_\eps(t,x))|^2\psi(t,x)\,dx\,dt\Big]\notag \\
  \le  & \E \Big[ {\bf 1}_{B} \int_{\R^d} \beta(u_\eps(0,x))\psi(0,x)\,dx\Big] 
   -\eps \E \Big[{\bf 1}_{B}\int_{\Pi_T}\beta^\prime(u_\eps(t,x))\grad u_\eps(t,x) \cdot \grad \psi(t,x)\,dx\,dt\Big]\notag \\
 +&  \E \Big[ {\bf 1}_{B}  \int_{\Pi_T} \Big(\beta(u_\eps(t,x)) \partial_t\psi(t,x) + \nu(u_\eps(t,x))\Delta \psi(t,x)
  - \grad \psi(t,x)\cdot \zeta(u_\eps(t,x)) \Big)\,dx\,dt\Big]\notag \\
 + &  \E \Big[ {\bf 1}_{B} \int_0^T \int_{E} \int_{\R^d} \int_0^1 \eta(x,u_\eps(t,x);z)\beta^\prime (u_\eps(t,x) + \theta\,\eta(x,u_\eps(t,x);z))
 \psi(t,x)\,d\theta\,dx\,\tilde{N}(dz,dt) \Big] \notag \\
+&  \E \Big[ {\bf 1}_{B} \int_0^T \int_{E} \int_{\R^d}  \int_0^1  (1-\theta)\eta^2(x,u_\eps(t,x);z)\beta^{\prime\prime} (u_\eps(t,x) + \theta\,\eta(x,u_\eps(t,x);z)) \notag \\
& \hspace{5cm} \times \psi(t,x)\,d\theta\,dx\,m(dz)\,dt \Big] \label{viscous-measure-inequality}
\end{align}
Let the predictable process $u(t,x)$ be the pointwise limit of $u_{\eps}(t,x)$ for a.e. $(t,x)\in (0,T)\times \R^d$ almost surely.
One can now pass to the limit in \eqref{viscous-measure-inequality} (cf. same argument as in \cite{BisMajKarl_2014}) except the
first term. The pointwise limit of $u_{\eps}(t,x)$ is not enough to pass the limit in the first term of the inequality 
because $u_\eps$ is in a gradient term. For this, we proceed as follows: fix $v\in L^2(\Omega \times \Pi_T)$. Define
\begin{align*}
f_\eps = \sqrt{\beta^{\prime\prime}(u_\eps(t,x))\psi(t,x) {\bf 1}_{B}},\quad \text{and} \quad g_\eps=\grad G(u_\eps(t,x)).
 \end{align*}
 Note that, $f_\eps$ is uniformly bounded and $g_\eps \rightharpoonup g=\grad G(u(t,x))$ in $L^2(\Omega \times \Pi_T)$. 
  Also, $f_\eps$ converges to $f$ pointwise (up to a subsequence) where $f=\sqrt{\beta^{\prime\prime}(u(t,x))\psi(t,x) {\bf 1}_{B}}$.
Since $|f_\eps\,v| \le  \sqrt{||\beta^{\prime\prime}||_{\infty} \psi(t,x)} |v(t,x)|$ and right hand side is $L^2$ integrable,
one can apply dominated convergence theorem to conclude 
\begin{align*}
 f_\eps\,v \longrightarrow fv \quad \text{in}\quad L^2(\Omega \times \Pi_T). 
\end{align*}
Moreover, we have $ f_\eps\,g_\eps \rightharpoonup fg \quad \text{in}\quad L^2(\Omega \times \Pi_T)$ and therefore, by Fatou's lemma for weak convergence, 
 \begin{align*}
  & \E \Big[ {\bf 1}_{B} \int_{\Pi_T} \beta^{\prime\prime}(u(t,x)) |\Grad G(u(t,x))|^2\psi(t,x)\,dx\,dt\Big]\notag \\ 
  \le  & \liminf_{\eps \downarrow 0}\, \E \Big[ {\bf 1}_{B} \int_{\Pi_T} \beta^{\prime\prime}(u_\eps(t,x))
  \big|\Grad G(u_\eps(t,x))\big|^2\psi(t,x)\,dx\,dt\Big].
 \end{align*}
 Thus, we can pass to the limit in \eqref{viscous-measure-inequality} as $\eps \goto 0$ and arrive at following inequality.
 \begin{align}
 & \E \Big[ {\bf 1}_{B} \int_{\Pi_T} \beta^{\prime\prime}(u(t,x)) |\Grad G(u(t,x))|^2\psi(t,x)\,dx\,dt\Big]
 -  \E \Big[ {\bf 1}_{B} \int_{\R^d} \beta(u_0(x))\psi(0,x)\,dx\Big] \notag \\
  \le  &  \E \Big[ {\bf 1}_{B}  \int_{\Pi_T} \Big(\beta(u(t,x)) \partial_t\psi(t,x) + \nu(u(t,x))\Delta \psi(t,x)
  - \grad \psi(t,x)\cdot \zeta(u_\eps(t,x)) \Big)\,dx\,dt\Big]\notag \\
 + &  \E \Big[ {\bf 1}_{B} \int_0^T \int_{E} \int_{\R^d} \int_0^1 \eta(x,u(t,x);z)\beta^\prime (u(t,x) + \theta\,\eta(x,u(t,x);z))
 \psi(t,x)\,d\theta\,dx\,\tilde{N}(dz,dt) \Big] \notag \\
+&  \E \Big[ {\bf 1}_{B} \int_0^T \int_{E} \int_{\R^d}  \int_0^1  (1-\theta)\eta^2(x,u(t,x);z)\beta^{\prime\prime}
(u(t,x) + \theta\,\eta(x,u(t,x);z)) \notag \\
& \hspace{5cm} \times \psi(t,x)\,d\theta\,dx\,m(dz)\,dt \Big]. \label{viscous-measure-inequality-1}
\end{align}
\vspace{.2cm}

We are now in a position to prove the existence of entropy solution for the underlying problem \eqref{eq:levy_stochconservation_laws}.
\begin{proof}[Proof of the Theorem \ref{thm:existenc}] The uniform moment estimate \eqref{bounds:a-priori-viscous-solution} together
with a general version of Fatou's lemma gives 
 \begin{align}
 \sup_{0\le t\le T} \E\Big[|| u(t,\cdot)||_2^2\Big] < \infty \,\quad \text{and}\,\quad \|\grad G(u)\|_{L^2(\Omega\times\Pi_T)}^2 <\infty \notag.
\end{align}
 For any  $ 0\le \psi\in C_c^\infty([0, \infty)\times \R^d)$ and 
 given convex entropy flux triple $(\beta,\zeta,\nu)$, Eq. \eqref{viscous-measure-inequality-1} holds for every 
 $B\in \mathcal{F}_T$. Hence, the following inequality 
\begin{align*}
& \int_{\R^d} \beta(u_0(x))\psi(0,x)\,dx  +  \int_{\Pi_T} \beta(u(t,x)) \partial_t\psi(t,x) \,dx\,dt\notag \\
& + \int_{\Pi_T} \nu(u(t,x))\Delta \psi(t,x)\,dx\,dt -  \int_{\Pi_T} \grad \psi(t,x)\cdot \zeta(u(t,x))\,dx\,dt \notag \\
&+ \int_0^T \int_{E} \int_{\R^d} \int_0^1 \eta(x,u(t,x);z)\beta^\prime \big(u(t,x) + \theta\,\eta(x,u(t,x);z)\big)\psi(t,x)\,d\theta\,dx\,\tilde{N}(dz,dt) \notag \\
& + \int_{E}  \int_{\Pi_T}  \int_0^1  (1-\theta)\eta^2(x,u(t,x);z)\beta^{\prime\prime} \big(u(t,x) + \theta\,\eta(x,u(t,x);z)\big)
       \psi(t,x)\,d\theta\,dx\,dt\,m(dz) \notag \\
& \quad \ge \int_{\Pi_T} \beta^{\prime\prime}(u(t,x)) \big|\Grad G(u(t,x))\big|^2\psi(t,x)\,dx\,dt
\end{align*}
holds $P$-almost surely. This shows that $u(t,x)$ is an entropy solution of \eqref{eq:levy_stochconservation_laws} in the sense of Definition \ref{defi:stochentropsol}.
This completes the proof.
\end{proof}

 We now close this section with a sketch of the justification of  our claim in Remark \ref{lem:p-bounds}. To see this, 
  let $h_\delta$ denote a smooth even convex approximation of $|.|^p$ define for positive $x$ by:
  
  $h_\delta$ vanishes at $0$  and uniquely recovered from its second order derivative defined as 
\\
\,  $h_\delta^{\prime\prime}(x)=x^{p-2}$ is $x \in[0,\frac1\delta]$ and $\frac{1}{\delta^{p-2}}$ if $x>\frac1\delta$.
\\
 It holds that , $0 \leq h_\delta(x) \nearrow h(x)=K_p|x|^p$ and there exists $C_p$ such that $0 \leq h_\delta^{\prime\prime}(x) \leq C_p h(x)$. Furthermore, it is easily seen that $ h_{\delta}^{\prime\prime}(x+y)\le\tilde{C}_p\big( h_{\delta}^{\prime\prime}(x)+h_{\delta}^{\prime\prime}(y)\big)$. 
 \medskip
 \\ 
 Note that the weak It\^{o}-L\'{e}vy formula in Theorem \ref{thm:weak-its} makes sense for $\beta =h_{\delta}$, as $h^{\prime\prime}_\delta$ is bounded.  This enables us write,  for almost every $t > 0$,
 
\begin{align*}
&E\int_{\R^d}h_\delta(u_\epsilon)dx - E\int_{\R^d}h_\delta(u_0)dx + E\int_0^t\int_{\R^d} (\phi^\prime(u_\epsilon)+\epsilon)h^{\prime\prime}_\delta(u_\epsilon)|\nabla u_\epsilon|^2\,dx\,dt 
\\&
= E\int_0^t\int_{E}\int_{\R^d} \Big(h_{\delta}(u_\epsilon +\eta(x,u_\epsilon;z ))-h_{\delta}(u_\epsilon)- \eta(x,u_\epsilon;z ) h_{\delta}^\prime(u_\epsilon)\Big)\,dx\, m(dz)\,dt.\\
&=  E\int_0^t\int_{E}\int_{\R^d} \int_0^1 (1-\theta)(\eta(x,u_\epsilon;z))^2 h_{\delta}^{\prime\prime}(u_\epsilon+ \theta \eta(x,u_\epsilon;z))\,d\theta \,dx\,m(dz)\,ds.
\end{align*}

We can now use the properties of $h_{\delta}$ and the assumptions on $\eta$ to arrive at
\begin{align*}
E\int_{\R^d}h_\delta(u_\epsilon)dx \leq& E\int_{\R^d}h_\delta(u_0)dx  +C_\eta E\int_0^t\int_{\R^d}(1+ u_\epsilon^2) h^{\prime\prime}_\delta(u_\epsilon) ds\\
\leq& E\int_{\R^d}|u_0|^pdx  +K_\eta E\int_0^t\int_{\R^d} h_\delta(u_\epsilon) ds
\end{align*}
and, by a weak Gronwall inequality, 
$E\int_{\R^d}h_\delta(u_\epsilon)dx \leq e^{K_\eta t} E\int_{\R^d}|u_0|^pdx$ for all almost all $t$.  This implies  $E\int_{\R^d}|u_\epsilon|^pdx \leq e^{C_\eta t} E\int_{\R^d}|u_0|^pdx$ by monotone convergence theorem. The solution $u$ will inherit the same property by Fatou's lemma. 
\medskip
\\
If $u_0$ is bounded and $\eta(x,u;z)=0$ for $|u|\geq M$, $M$ been given, then, consider non-negative regular convex function  $x \mapsto h(x)=[(x+K)^-]^2+[(x-K)^+]^2$ where $K=\max(M+M_1,\|u_0\|_\infty)$. Since $h(u_0)=0$ and $h$ vanishes where $\eta$ is active, the It\^o formula Yields $E\int_{\R^d}|h(u_\epsilon)|dx = 0$ and $u_\epsilon$ is uniformly bounded by $K$. Again, the solution $u$ will inherit the same property by passing to the limit.
}

\section{Uniqueness of entropy solution}
To prove the uniqueness of entropy solution, we compare any entropy solution to the viscous solution \textit{via} Kruzkov's doubling variables method 
and then pass to the limit as viscous parameter goes to zero. We have already shown that limit of the viscous solutions serve for existence of entropy 
solution for the underlying problem. Now let $v(t,x)$ be any entropy solution and $u_{\eps}(t,x)$ be viscous solution for the problem 
\eqref{eq:levy_stochconservation_laws-viscous}. Then one can use  exactly the same argument as in Section \ref{sec:existence-entropy}, and 
end up with the following equality 
\begin{align*}
  \E \Big[\int_{\R^d} \int_0^1 \big|u(t,x,\alpha)-v(t,x)|\phi_n(x)\,d\alpha\,dx\Big]=0.
 \end{align*} 
This implies that, for almost every $t\in [0,\infty)$, $ v(t,x)=u(t,x,\alpha)$ for almost every $x \in \R^d$, 
$(\omega,\alpha) \in \Omega \times (0,1)$. In other words, this proves the uniqueness for entropy solutions. 

\appendix
\section{Weak It\^{o}-L\'{e}vy formula}
Let $u$ be a $H^1(\R^d)$-valued $\mathcal{F}_t$-predictable process and it is a weak solution to  the SPDE
\begin{align}
  du(t,x) -\Delta \phi(u(t,x))\,dt =& \mbox{div}_x f(u(t,x)) \,dt +\int_{E} \eta(x, u(t,x);z)\tilde{N}(dz,dt) \notag \\
 & \qquad \quad  + \eps \Delta u(t,x)\,dt,\quad t>0, ~ x\in \R^d.\label{eq:levy_stochconservation_laws-viscous-appendix}
\end{align}In addition, in view of \eqref{convergence:weak-1}, we further assume that $u \in L^2\big((0,T)\times \Omega;H^1(\R^d)\big)$. Moreover, $u$ satisfies the initial condition $u_0\in L^2(\R^d)$ in the following sense: $P$ -almost surely
\begin{align}
 \label{weak-initial-consition} \lim_{h\rightarrow 0}\frac 1h  \int_0^h \int_{\R^d}  u(t,x) \phi(x)\, dx = \int_{\R^d} u_0(x) \phi(x ) \, dx. 
\end{align} for every $\phi \in C_c^{\infty}(\R^d)$.  We have the following weak version of It\^{o}-L\'{e}vy formula for $u(t,\cdot)$. 
\begin{thm} \label{thm:weak-its}Let the assumptions \ref{A1}-\ref{A5} hold and $u(t,\cdot)$ be a $H^1(\R^d)$-valued weak solution of \eqref{eq:levy_stochconservation_laws-viscous-appendix}, as described in subsection \ref{Existence of weak solution}., which satisfies \eqref{weak-initial-consition}. Then for every entropy triplet $(\beta, \zeta, \nu)$ and $\psi\in C_{c}^{1,2}([0,\infty)\times
 \R^d)$, it holds $P$-almost surely that
 \begin{align}
  &\int_{\R^d} \beta(u(T,x))\psi(T,x)\,dx -  \int_{\R^d} \beta(u(0,x))\psi(0,x)\,dx  \notag \\
  = & \int_{\Pi_T} \beta(u(t,x)) \partial_t\psi(t,x) \,dx\,dt -  \int_{\Pi_T} \grad \psi(t,x)\cdot \zeta(u(t,x))\,dx\,dt 
   \notag \\
 + & \int_{\Pi_T} \int_{E}  \int_0^1 \eta(x,u(t,x);z)\beta^\prime (u(t,x)
 + \theta\,\eta(x,u(t,x);z))\psi(t,x)\,d\theta\,\tilde{N}(dz,dt)\,dx \notag \\
+&\int_{\Pi_T}\int_{E}  \int_0^1  (1-\theta)\eta^2(x,u(t,x);z)\beta^{\prime\prime} (u(t,x) + \theta\,\eta(x,u(t,x);z))
\psi(t,x)\,d\theta\,m(dz)\,dx\,dt \notag \\
 -&   \int_{\Pi_T}  \Big(\eps \nabla \psi(t,x).\nabla_x\beta(u(t,x)) +\eps \beta''(u(t,x))|\nabla_x u(t,x)|^2\psi(t,x)\Big)\,dx\,dt \notag \\
-&  \int_{\Pi_T} \phi^\prime(u(t,x)) \beta^{\prime\prime}(u(t,x)) \big|\grad u(t,x)\big|^2\psi(t,x)\,dx\,dt
+  \int_{\Pi_T} \nu(u(t,x))\Delta \psi(t,x)\,dx\,dt\notag
\end{align} for almost every $T> 0$.
 \end{thm}
 
 \begin{proof}
Let $\{\tau_k\}$ be a standard sequence of mollifiers on $\R^d$. Then for every $\rho(\cdot)\in C_c^1((0, T))$ we  have
 \begin{align}
   \notag-\int_0^T u(s,\cdot )*\tau_k \rho^{\prime}(s)\,ds =& \int_0^T \rho(s) \Delta(\phi(u(s,\cdot))* \tau_k)\,ds + \int_0^T \rho(s)\, \mbox{div}_x (f(u)\con \tau_\kappa) \,ds\\
   & + \int_0^T\int_E \rho(s)\big(\eta(x,u, z)*\tau_k\big)\tilde{N}(\,dz,\, ds)+\epsilon \int_0^T \Delta \big( u * \tau_k(s,x)\big)\rho(s)\,ds\label{eq:mollified}
 \end{align} holds $P$-almost surely.  For every $n\in \mathbb{N}$, define
 
 $$\rho_n^t(s)=
\begin{cases}
ns ~\text{if}~0\le s\le \frac 1n\\
1, \quad \text{if} ~\frac 1n \le s<  t\\
1-n(s-t),\quad \text{if}~t+\frac 1n>s\ge  t\\
0,\quad \text{elsewhere}.
\end{cases}$$ It follows by standard approximation argument that \eqref{eq:mollified} is still valid if we replace $\rho(\cdot)$ by $\rho_n^t(\cdot)$. Afterwards, we invoke right continuity of stochastic integral  and standard facts related to Lebesgue points of Banach space valued functions to pass to the limit $n\rightarrow \infty $ and conclude for almost all $t>0$

 \begin{align}
\notag u*\tau_k(t,\cdot) - u_0*\rho_k =& \int_0^t \Delta(\phi(u(s,\cdot))* \tau_k)\,ds + \int_0^t \mbox{div}_x (f(u)\con \tau_\kappa) \,ds\\
   & + \int_0^t\int_E \big(\eta(x,u, z)*\tau_k\big)\tilde{N}(\,dz,\, ds)+\epsilon \int_0^t \Delta \big( u * \tau_k(s,x)\big)\,ds\label{eq:mollified-SDE}
 \end{align} $P$-holds almost surely.  In the above, we have used that the weak solution satisfies the initial condition in the sense of \eqref{weak-initial-consition}. Let $\beta$ be the entropy function mentioned in the statement and  $\psi$ be the test function specified. Now we apply It\^{o}-L\'{e}vy chain rule to $\beta(u*\tau_k(t,\cdot))$ to have, for almost every $t>0$,
  \begin{align}
\notag \beta\big(u*\tau_k(t,\cdot)\big) &= \beta(u_0*\rho_k)+ \int_0^t \beta^\prime\big(u*\tau_k(t,\cdot)\big) \Delta(\phi(u(s,\cdot))* \tau_k)\,ds\\ &+ \int_0^t\beta^\prime\big(u*\tau_k(t,\cdot)\big)\mbox{div}_x (f(u)\con \tau_\kappa) \,ds+\epsilon \int_0^t \beta^\prime\big(u*\tau_k(t,\cdot)\big)\Delta \big( u * \tau_k(s,x)\big)\,ds\notag\\
    &+ \int_0^t\int_E \Big(\beta(u*\tau_k+ \eta(x, u, z)*\tau_k)-\beta(u*\tau)\Big)\tilde{N}(\,dz,\, ds)\notag\\
    &+ \int_0^t\int_E \Big(\beta(u*\tau_k+ \eta(x, u, z)*\tau_k)-\beta(u*\tau) -\eta(x, u, z)*\tau_k\beta^{\prime}(u*\tau_k)\Big)\,m(\,dz)\,dt,
    \label{eq:mollified-SDE-ito}   
      \end{align} P-almost surely. We now apply It\'{o}-L\'{e}vy product rule on $\beta(u*\tau_k)\psi(t,x)$ and integrate with respect to $x$ to obtain for almost every $T> 0$, 
      
        \begin{align}
\notag& \int_{\R^d} \beta\big(u*\tau_k(T,x)\big)\,\psi(T,x)\,dx \\\notag &= \int_{\R^d}\beta(u_0*\rho_k)\psi(0,x)\,dx+\int_0^T\int_{\R^d} \beta(u*\tau_k)\partial_s\psi(s,x)\,dx\,ds\\\notag&+ \int_0^T\int_{\R^d} \beta^\prime\big(u*\tau_k(s,\cdot)\big) \Delta(\phi(u(s,\cdot))* \tau_k)\psi(s,x)\,dx\,ds\\&\notag+ \int_0^T\int_{\R^d}\beta^\prime\big(u*\tau_k(s,\cdot)\big)\mbox{div}_x (f(u)\con \tau_\kappa)\psi(s,x) \,dx\,ds\\&+\epsilon \int_0^T\int_{R^d} \psi(s,x)\beta^\prime\big(u*\tau_k(s,\cdot)\big)\Delta \big( u * \tau_k(s,x)\big)\,dx\,ds\notag\\
    &+ \int_0^T\int_{\R^d}\int_E\psi(s,x) \Big(\beta(u*\tau_k+ \eta(x, u, z)*\tau_k)-\beta(u*\tau)\Big)\,dx\tilde{N}(\,dz,\, ds)\notag\\
    &+ \int_0^T\int_{\R^d}\int_E\psi(s,x) \Big(\beta(u*\tau_k+ \eta(x, u, z)*\tau_k)-\beta(u*\tau) -\eta(x, u, z)*\tau_k\beta^{\prime}(u*\tau_k)\Big)\,m(\,dz)\,dx \,ds,
    \label{eq:mollified-SDE-ito-2}   
      \end{align} almost surely.        
      Note that $ u*\tau_k(T,\cdot) \rightarrow u(T,\cdot)$  and $u_0*\tau_k\rightarrow u_0$ in $L^2(\Omega\times \R^d)$ as $k \rightarrow 0$. Therefore by Lipschitz continuity of $\beta$, we have $ \int_{\R^d} \beta\big(u*\tau_k(T,x)\big)\,\psi(T,x)\,dx \rightarrow \int_{\R^d} \beta\big(u(T,x)\big)\,\psi(T,x)\,dx$ and $ \int_{\R^d}\beta(u_0*\rho_k)\psi(0,x)\,dx \rightarrow  \int_{\R^d}\beta(u_0)\psi(0,x)\,dx$ in $L^2(\Omega)$. By a similar reasoning, $\int_0^T\int_{\R^d} \beta(u*\tau_k)\partial_s\psi(s,x)\,dx\,ds \rightarrow \int_0^T\int_{\R^d} \beta(u)\partial_s\psi(s,x)\,dx\,ds$ as $k\rightarrow 0$. 
      
      Furthermore, note that 
      \begin{align*}
      \int_0^T\int_{\R^d} \beta^\prime\big(u*\tau_k(t,\cdot)\big) \Delta(\phi(u(s,\cdot))* \tau_k)\psi(s,x)\,ds\,dx\\
    =-\int_0^T\int_{\R^d} \nabla_x \Big(\psi(t,x)\beta^\prime\big(u*\tau_k(t,x)\big)\Big).\Big( \nabla\phi(u())* \tau_k)(s,x)\Big)\,dx\,ds
                \end{align*} and $\nabla u, \nabla \phi(u) \in L^2\Big(0,T; L^2(\Omega\times \R^d)\Big)$. 
                
                Therefore, $ \nabla_x \big(\psi(t,x)\beta^\prime\big(u*\tau_k(t,x)\big)\big) \rightarrow \nabla_x \big(\psi(t,x)\beta^\prime\big(u(t,x)\big)\big) $ and 
 $\nabla\phi(u)* \tau_k \rightarrow  \nabla\phi(u)$ in $L^2\big(0,T; L^2(\Omega\times \R^d)\big)$ as $k\rightarrow 0$. Therefore, $  \int_0^T\int_{\R^d} \beta^\prime\big(u*\tau_k(t,\cdot)\big) \Delta(\phi(u(s,\cdot))* \tau_k)\psi(s,x)\,ds\,dx \rightarrow -\int_0^T\int_{\R^d} \nabla_x \Big(\psi(t,x)\beta^\prime\big(u(t,x)\big)\Big).\Big( \nabla\phi(u(s,x)\Big)\,dx\,ds $\\ in $L^1(\Omega)$ as $k \rightarrow 0$. By the same reasoning, 
 
  $$ \int_{\Pi_T} \psi(s,x)\beta^\prime\big(u*\tau_k(s,x)\big)\Delta \big( u * \tau_k(s,x)\big)\,dx\,ds \rightarrow - \int_{\Pi_T}\nabla_x\big( \psi(s,x)\beta^\prime(u(s,x))\big).\nabla \big( u (s,x)\big)\,dx\,ds.$$ in $ L^1(\Omega)$ as $k \rightarrow 0$.

 Also, it may be recalled that $\text{div}_x f(u)\in L^2\big(0,T; L^2(\Omega\times \R^d)\big)$ and $\beta(u*\tau_k)\rightarrow \beta(u)$ in $ L^2\big(0,T; L^2(\Omega\times \R^d)\big)$ as $\kappa \rightarrow 0$. Therefore,
 \begin{align*}
  \int_0^T\int_{\R^d}\beta^\prime\big(u*\tau_k(t,\cdot)\big)\mbox{div}_x (f(u)\con \tau_\kappa)\psi(s,x) \,dx\,ds
  \rightarrow \int_0^T\int_{\R^d}\beta^\prime\big(u\big)\mbox{div}_x (f(u))\psi(s,x) \,dx\,ds 
    \end{align*} as $k\rightarrow 0$ in $L^1(\Omega)$. To this end, we denote 
    
    \begin{align*}
    I_k(s,z)&= \int_{\R^d}\psi(s,x) \Big(\beta(u*\tau_k+ \eta(x, u, z)*\tau_k)-\beta(u*\tau)\Big)\,dx,\\
    I(s,z)&=  \int_{\R^d}\psi(s,x) \Big(\beta(u+ \eta(x, u, z))-\beta(u*\tau)\Big)\,dx.   
        \end{align*}
        It follows from straightforward computation that $\int_0^T \int_E |I_k(s,z)-I(s,z)|^2m(dz)\,ds \rightarrow 0$ as $k\rightarrow 0$. Therefore, we can invoke It\^{o}-L\'{e}vy isometry and pass to the limit $k\rightarrow 0$, in the martingale term in \eqref{eq:mollified-SDE-ito-2}.  This competes the validation of passage to the limit as $k\rightarrow 0$  in every term of \eqref{eq:mollified-SDE-ito-2}. The assertion is now concluded by simply letting $k \rightarrow 0$ in    
        \eqref{eq:mollified-SDE-ito-2} and rearranging the terms.

 \end{proof}


\end{document}